\documentclass[pdf, 10pt]{amsart}

\usepackage{amsthm, amsmath, amssymb}
\usepackage[OT1]{fontenc}
\usepackage[colorlinks,citecolor=blue,urlcolor=blue]{hyperref}
\usepackage{color}

\newtheorem{thm}[equation]{Theorem}

\newtheorem{defn}[equation]{Definition}

\newtheorem{lem}[equation]{Lemma}

\def\R{{\mathbb{R}}}

\def\C{{\mathbb{C}}}
\def\N{{\mathbb{N}}}
\def\Z{{\mathbb{Z}}}

\renewcommand{\a}{\alpha}
\renewcommand{\b}{\beta}
\newcommand{\g}{\gamma}

\renewcommand{\d}{\delta}

\newcommand{\eps}{\epsilon}
\newcommand{\e}{\varepsilon}

\renewcommand{\t}{\tau}

\newcommand{\te}{\theta}
\newcommand{\s}{\sigma}

\newcommand{\vp}{\varphi}

\newcommand{\8}{\infty}

\newcommand{\vt}{\vartheta}
\newcommand{\vr}{\varrho}

\newcommand{\ep}[1]{e^{2\pi i#1}}
 \numberwithin{equation}{section}

\begin{document}

\title[] {$\ell^p(\mathbb{Z})$ -- boundedness of discrete maximal functions along thin subsets of primes and pointwise ergodic theorems}
\author[M.Mirek]
{Mariusz Mirek}
\address{M.Mirek \\
Universit\"{a}t Bonn \\
Mathematical Institute\\
Endenicher Allee 60\\
D--53115 Bonn \\
Germany \&\\
Uniwersytet Wroc{\l}awski\\
Mathematical Institute \\
Plac Grunwaldzki 2/4\\
 50--384 Wroc{\l}aw\\
 Poland}
 \email{mirek@math.uni.bonn-de}

\thanks{
The author was supported by NCN grant DEC--2012/05/D/ST1/00053}

\maketitle

\begin{abstract}
We establish the first pointwise ergodic theorems
along thin sets of prime numbers; a set with zero density
with respect to the primes. For instance we will be able
to achieve this with the Piatetski--Shapiro primes. Our
methods will be robust enough to solve the ternary Goldbach
problem for some thin sets of primes.
\end{abstract}

\section{Introduction and statement of results}

In the middle 1980's,  Bourgain and  Wierdl generalized
Birkhoff's pointwise ergodic theorem, showing that almost everywhere
convergence still holds when averages are taken only along the set of prime
numbers $\mathbf{P}$. More precisely, let $(X, \mathcal{B}, \mu, T)$ be a general dynamical
system where $T$ is an invertible,  measure preserving transformation
on the $\sigma$--finite measure space $(X, \mathcal{B}, \mu)$. Then, for any $ f\in L^r(X,\mu)$ where $r>1$, the limit
\begin{align}\label{ergprim}
 \lim_{N\to\infty} \frac{1}{|\mathbf{P}\cap[1, N]|}  \sum_{p\in \mathbf{P}\cap[1, N]} f(T^p x),
\end{align}
exists for $\mu$--almost every $x\in X$.
See \cite{B3} and \cite{Wie}. At almost the same time Nair \cite{Na1, Na2} showed that \eqref{ergprim} remains still valid when $T^p$ is replaced with $T^{W(p)}$, where $W$ is an arbitrary integer valued polynomial.
The restriction to the range $r>1$ is essential as
LaVictoire \cite{LaV1}, extending work of Buczolich and Mauldin \cite{BM}, showed
the pointwise convergence result \eqref{ergprim} fails on $L^1(X,\mu)$.

  Since the work of Bourgain and Wierdl, there have been many results
   establishing both  pointwise ergodic theorems along various arithmetic subsets
of the integers and investigating discrete
analogues of classical operators with arithmetic features; see
\cite{BKQW, IMSW, IW,  MSW, M2, O, P, RW, SW1, SW2, SW3, SW4, UZ, W}. However, not many have been proved for the subsets of primes; see \cite{Na1, Na2, Wie} and recently \cite{MT}.

The  aim of this article is to extend the result of Bourgain \cite{B3}, Wierdl \cite{Wie} and Nair \cite{Na1, Na2} to the case where the ergodic averages operators \eqref{ergprim} are defined along appropriate subsets of primes.

One of the main objects will be the set of Piatetski--Shapiro primes $\mathbf{P}_{\g}$
of fixed type $\g<1$, where $\g$ is sufficiently close to $1$, i.e.
$$\mathbf{P}_{\g}=\{p\in\mathbf{P}: \exists_{n\in\N}\
p=\lfloor n^{1/{\g}}\rfloor\}.$$

\noindent In particular we will show  the pointwise ergodic theorem along the set of Piatetski--Shapiro primes.
\begin{thm}\label{ergo1}
Let $(X, \mathcal{B}, \mu, T)$ be a dynamical system and let $W:\Z\mapsto\Z$ be a fixed  polynomial of degree $q\in\N$. Then there exists $0<\g_q<1$ such that for every $\g_q<\g<1$ the limit
\begin{align*}
 \lim_{N\to\infty} \frac{1}{|\mathbf{P}_{\g}\cap[1, N]|}  \sum_{p\in \mathbf{P}_{\g}\cap[1, N]} f(T^{W(p)} x),
\end{align*}
exists $\mu$--almost everywhere on $X$, for every $f\in L^r(X, \mu)$ with $r>1$.
\end{thm}

In view of the transference principle, one can transpose our problem and work with the set of integers
rather than an abstract measure space $X$. In these settings we will consider the averages
\begin{align}\label{maxs}
 M_{S, N}f(x)=\frac{1}{|S\cap[1, N]|}\sum_{p\in
S\cap[1, N]}f(x-W(p)),\ \ \mbox{for \ $x\in\Z$},
\end{align}
along a fixed set $S\subseteq \mathbf{P}$, where $W:\Z\mapsto\Z$ is a fixed  polynomial of degree $q\in\N$. Finally, let $M_{S}f(x)=\sup_{N\in\N}|M_{S, N}f(x)|$ be the maximal function corresponding with the averages defined in \eqref{maxs}. Such maximal functions, as we will see in the sequel, play essential roles in the pointwise convergence problems.

It is worth emphasizing that, in fact,
Bourgain \cite{B3} and Wierdl \cite{Wie} (with $W(x)=x$)  and Nair \cite{Na1, Na2} (with general polynomials $W$) showed that the maximal function $M_{\mathbf{P}}$ is bounded on $\ell^r(\Z)$ for every $r>1$.

We shall show the following maximal theorem along the set of Piatetski--Shapiro primes.
\begin{thm}
Let $W:\Z\mapsto\Z$ be a fixed  polynomial of degree $q\in\N$. Then there exists $0<\g_q<1$ such that for every $\g_q<\g<1$ and every $1<r\le \8$ there is a constant $C>0$ such that
\begin{align*}
  \|M_{\mathbf{P}_{\g}}f\|_{\ell^r(\Z)}\le C\|f\|_{\ell^r(\Z)},
\end{align*}
for every $f\in\ell^r(\Z)$.
\end{thm}
Throughout the paper we will be considering thin subsets of primes. The set $S$ consisting of prime numbers is called thin if $|S|=\8$ and
$$|S\cap[1, x]|=\mathrm{o}(\mathbf{P}\cap[1, x]) \ \ \mbox{as \ $x\to\8$}.$$
The famous theorem of Piatetski--Shapiro \cite{PS} establishes the asymptotic formula
 $$|\mathbf{P}_{\g}\cap[1, x]|\sim\frac{x^{\g}}{\log x}  \ \ \mbox{as \ $x\to\8$},$$ for every $\g\in(11/12, 1)$, which obviously implies that $\mathbf{P}_{\g}$ is a thin subset of primes. The asymptotic formula was discovered by Piatetski--Shapiro \cite{PS} in 1953 for $\gamma\in(11/12, 1)$. This range was subsequently improved  by Kolesnik \cite{Kol}, Graham (unpublished), Leitmann (unpublished), Heath--Brown \cite{HB}, Kolesnik \cite{Kol1}, Liu--Rivat \cite{LR}.  Currently, due to recent result of Rivat and Sargos \cite{RS}, we know that $\gamma\in(2426/2817, 1)$, and this is the best known range.

In fact the set of Piatetski--Shapiro will be a particular example of a wide family of thin subsets of $\mathbf{P}$ of the form
\begin{align}\label{defnph}
 \mathbf{P}_{h}=\{p\in\mathbf{P}:
\exists_{n\in\N}\ p=\lfloor h(n)\rfloor\},
\end{align}
 where $h$ is an appropriate function, as in Definition \ref{defn}. Specifically, we will study the sets $\mathbf{P}_{h}$ with functions $h$ of the following form
\begin{align*}
  h_1(x)=x^c\log^Ax,\ \ h_2(x)=x^ce^{A\log^Bx},\ \ h_3(x)=x\log^Cx, \ \ h_4(x)=xe^{C\log^Bx}, \ \ h_5(x)=xl_m(x),
\end{align*}
where $c\in(1, 2)$, $A\in\R$, $B\in(0, 1)$, $C>0$, $l_1(x)=\log x$ and $l_{m+1}(x)=\log(l_m(x))$, for $m\in\N$.

However, we encourage the reader to bear in mind the set of Piatetski--Shapiro primes $\mathbf{P}_{\g}$
as a principal example. This will allow us to get a better understanding of further generalizations.

\begin{defn}\label{defn}
Let $c\in[1, 2)$ and $\mathcal{F}_c$ be the family of all functions $h:[x_0, \8)\mapsto [1, \8)$ (for some $x_0\ge1$)  satisfying
\begin{enumerate}
\item[(i)] $h\in \mathcal{C}^{\8}([x_0, \8))$ and
$$h'(x)>0,\ \ \ \ h''(x)>0, \ \ \mbox{for every \  $x\ge x_0$.}$$
\item[(ii)] There exists a real valued function $\vartheta\in\mathcal{C}^{\8}([x_0, \8))$ and a constant $C_h>0$ such that
\begin{align}\label{eq1}
  h(x)=C_hx^c\ell_h(x), \ \ \mbox{where}\ \ \ell_h(x)=e^{\int_{x_0}^x\frac{\vt(t)}{t}dt}, \ \ \mbox{for every \  $x\ge x_0$,}
\end{align}
and if $c>1$, then for every $n\in\N$
\begin{align}\label{eq2}
\lim_{x\to\8}\vartheta(x)=0, \ \ \mbox{and} \ \ \lim_{x\to\8}x^n\vartheta^{(n)}(x)=0.
\end{align}
  \item[(iii)] If $c=1$, then $\vt(x)$ is positive, decreasing and  for every $\varepsilon>0$ there is a constant $C_{\e}>0$
  \begin{align}\label{eq3}
    \frac{1}{\vt(x)}\le C_{\varepsilon}x^{\varepsilon}, \ \ \mbox{and} \ \ \lim_{x\to\8}\frac{x}{h(x)}=0.
  \end{align}
  Furthermore, for every $n\in\N$
  \begin{align}\label{eq4}
  \lim_{x\to\8}\vartheta(x)=0, \ \ \mbox{and} \ \ \lim_{x\to\8}\frac{x^n\vartheta^{(n)}(x)}{\vt(x)}=0.
\end{align}
\end{enumerate}
\end{defn}
Let $\vp:[h(x_0), \8)\mapsto[1, \8)$ be the inverse function to $h$ and $\pi_h(x)$ denotes the cardinality of the set $\mathbf{P}_{h, x}=\mathbf{P}_{h}\cap[1, x]$. The family $\mathcal{F}_c$ was introduced by Leitmann in \cite{Leit}, where he showed
\begin{align}\label{asympto}
 \pi_h(x)\sim\frac{\vp(x)}{\log x}  \ \ \mbox{as \ $x\to\8$},
\end{align}
for every  $h\in\mathcal{F}_c$ with $c\in[1, 12/11)$.

 The family $\mathcal{F}_c$ was also considered by the author in \cite{M1}, where the counterpart of Roth's theorem for the sets $\mathbf{P}_{h}$ has been proved. See also in \cite{M2}.

Let $c_q=(2^{2q+2}+2^q-2)/(2^{2q+2}+2^q-3)$ for $q\in\N$.  Our main result is the following.
\begin{thm}\label{maxithm}
Let $W:\Z\mapsto\Z$ be a fixed polynomial of degree $q\in\N$. Assume that $h\in\mathcal{F}_c$ with $c\in[1, c_q)$. Then for every $1<r\le \8$ there is a constant $C>0$ such that
\begin{align}\label{maxiest}
  \|M_{\mathbf{P}_h}f\|_{\ell^r(\Z)}\le C\|f\|_{\ell^r(\Z)},
\end{align}
for every $f\in\ell^r(\Z)$.
\end{thm}
The proof of Theorem \ref{maxithm}, see Section \ref{sectmax}, will  be based to a certain extent on the ideas of Bourgain pioneered in
\cite{B1}, \cite{B2} and \cite{B3}, see also \cite{Na1, Na2, Wie}. However, the circle method of Hardy and Littlewood is inefficient here, (it was one of the main tools in Bourgain's works).
Instead of that the main basic idea in the proof is to make use of the following inequality
$$\|M_{\mathbf{P}_h}f\|_{\ell^r(\Z)}\le C\big(\|M_{\mathbf{P}}f\|_{\ell^r(\Z)}+\|\sup_{N\in\N}|(M_{\mathbf{P}_h, 2^N}-M_{\mathbf{P}, 2^N})f\|_{\ell^r(\Z)}\big).$$
In view of Bourgain--Wierdl's \cite{B3, Wie} and Nair's \cite{Na1, Na2} theorems the only point remaining is to bound the maximal function associated with the error term.
  Heuristically speaking, our aim will be to show, working on the Fourier transform side,  that
 $\|M_{\mathbf{P}_h, 2^N}-M_{\mathbf{P}, 2^N}\|_{\ell^2(\Z)\mapsto\ell^2(\Z)}=O(2^{-\d N})$ for some $\d>0$, see Lemma \ref{formlem} in Section \ref{sectformlem}. In order to establish this decay we will use some special Van der Corput's inequality rather than Weyl's inequality to bound the exponential sums corresponding with the Fourier transform of the operators $M_{\mathbf{P}_h, 2^N}-M_{\mathbf{P}, 2^N}$, see Section \ref{sectexp} and Section \ref{sectformlem}.
 As far as we know there are no other results dealing with maximal functions along thin subsets of primes.


 Theorem \ref{maxithm} combined with some oscillation inequality (see \eqref{ergosc} in Section \ref{secterg}) will lead us, via the transference principle, to the following generalization of Theorem \ref{ergo1}.
\begin{thm}\label{ergthm}
Let $(X, \mathcal{B}, \mu, T)$ be a dynamical system. Let $W:\Z\mapsto\Z$ be a fixed polynomial of degree $q\in\N$. Assume that $h\in\mathcal{F}_c$, with $c\in[1, c_q)$. Then, for every $f\in L^r(X, \mu)$ where $r>1$, the ergodic averages
\begin{align}\label{ergh}
  A_{h, N}f(x)=\frac{1}{\pi_h(N)}\sum_{p\in\mathbf{P}_{h, N}}f(T^{W(p)}x) \ \ \mbox{for  \  $x\in X$},
\end{align}
converge $\mu$--almost everywhere on $X$.
\end{thm}

On the other hand there is a natural question about the endpoint estimates for ${M}_{\mathbf{P}_h}f$, i.e. when $r=1$. Recently, Buczolich and Mauldin \cite{BM} and LaVictoire \cite{LaV1} showed, as it was mentioned above, that the pointwise convergence of ergodic averages along $p(n)=n^k$ for $k\ge2$ or the set of primes fails on $L^1$.
In view of the recent achievements in this field, it would be nice to know what  the answer is to this question in our case.

Finally, we will show that the ternary Goldbach problem has a solution
 in primes belonging to $\mathbf{P}_{h}$. The ternary Goldbach conjecture, or three--primes problem, asserts that every odd integer $N$ greater than $5$ is the sum of three primes, see for instance \cite{Nat}. This conjecture has been recently verified by Helfgott in the series of papers \cite{H1, H2, H3}.
 \begin{thm}\label{asymptthm}
Let $0<\g_1, \g_2, \g_3\le 1$ be fixed real numbers such that
 \begin{align}\label{asymptthmass}
 \begin{split}
   16(1-\g_1)&+14(1-\g_2)+14(1-\g_3)<1,\\
   16(1-\g_2)&+14(1-\g_1)+14(1-\g_3)<1,\\
   16(1-\g_3)&+14(1-\g_1)+14(1-\g_2)<1.
   \end{split}
 \end{align}
 Assume that $h_1\in\mathcal{F}_{1/\g_1}$, $h_2\in\mathcal{F}_{1/\g_2}$, $h_3\in\mathcal{F}_{1/\g_3}$ and $\vp_1, \vp_2, \vp_3$ be their inverse respectively.
 Then there exists a  constant $C(\g_1, \g_2, \g_3)>0$ such that $R(N)$ the number of representations  of an odd $N\in\N$ as a sum of three primes $p_i\in\mathbf{P}_{h_i}$
where $i=1, 2, 3$ satisfies
\begin{align}\label{asymptineq}
    R(N)=\sum_{\genfrac{}{}{0pt}{}{p_1+p_2+p_3=N}{p_i\in\mathbf{P}_{h_i, N}}}1
  \ge C(\g_1, \g_2, \g_3)\cdot\frac{\mathfrak{S}(N)\vp_1(N)\vp_2(N)\vp_3(N)}{N\ \log^3N},
  \end{align}
  for every sufficiently large $N\in\N$, where $\mathfrak{S}(N)=\prod_{p\in\mathbf{P}}\big(1-\frac{1}{(p-1)^3}\big)
    \prod_{p|N}\big(1-\frac{1}{p^2-3p+3}\big)>0$ is the singular series as in the classical ternary Goldbach problem.
 \end{thm}
 In particular  \eqref{asymptineq} means that every sufficiently large odd integer can be written as a sum of three primes $p_i\in\mathbf{P}_{h_i}$ where $i=1, 2, 3$. Balog and Friedlander in \cite{BF} and Kumchev in \cite{Kum} solved the ternary Goldbach problem in the Piatetski--Shapiro primes.

Here we extend the results of  Balog and Friedlander, and Kumchev to more general sets $\mathbf{P}_{h}$ for $h\in\mathcal{F}_c$ at the expanse of more restrictive conditions on $\g_1, \g_2, \g_3$ in \eqref{asymptthmass}. The proof of Theorem \ref{asymptthm} will be a combination of methods developed by Heath--Brown \cite{HB} with the  Vinogradov's methods \cite{Nat}.

The paper is organized as follows. In Section \ref{sectf} we give the necessary properties of function $h\in\mathcal{F}_c$ and its inverse $\vp$. In Section \ref{sectexp} we estimate some exponential sums which allow us to give the proof of Lemma \ref{formlem} in Section \ref{sectformlem}. In the last three sections  we give the proofs of Theorem \ref{maxithm}, Theorem \ref{ergthm} and Theorem \ref{asymptthm} respectively.

Throughout the whole paper,  $C > 0$ will  stand for a large positive constant whose value may vary from occurrence to occurrence.
For two quantities $A>0$ and $B>0$ we say that $A\lesssim B$ ($A\gtrsim B$) if there exists an absolute constant $C>0$ such that $A\le CB$ ($A\ge CB$). We will shortly write that $A\simeq B$, if $A\lesssim B$ and $A\gtrsim B$ hold simultaneously. We will also write $A\lesssim_{\d} B$ ($A\gtrsim_{\d} B$) to indicate that the constant $C>0$ depends on some $\d>0$.

\section*{Acknowledgements}
The author is greatly indebted to Christoph Thiele and Jim Wright for many helpful suggestions improving exposition of this paper.

\section{Basic properties of functions $h$ and $\vp$}\label{sectf}
Let us begin this section from reviewing  the properties of function $h\in\mathcal{F}_c$ and its inverse $\vp$ which have been established in \cite{M1}. It is easy to see that there exists a function $\te:[h(x_0),\8)\mapsto\R$ such that $\lim_{x\to\8}\te(x)=0$ and $x\vp'(x)=\vp(x)(\g+\te(x))$. Moreover,
  \begin{align}\label{funfi}
  \vp(x)=x^{\g}\ell_{\vp}(x),\ \ \ \mbox{where}\ \ \ \ell_{\vp}(x)=e^{\int_{h(x_0)}^x\frac{\te(t)}{t}dt+D},
\end{align}
  for every $x\ge h(x_0)$, where $D=\log x_0/h(x_0)^{\g}$ and
  \begin{align}\label{tetadef}
  \te(x)=\frac{1}{(c+\vartheta(\vp(x)))}-\g
  =-\frac{\vartheta(\vp(x))}{c(c+\vartheta(\vp(x)))}.
\end{align}
  If $L(x)=\ell_h(x)$ or $L(x)=\ell_{\vp}(x)$, then for every $\e>0$ we have
  \begin{align}\label{slowhfi}
    \lim_{x\to\8}x^{-\e}L(x)=0,\ \ \ \mbox{and}\ \ \ \lim_{x\to\8}x^{\e}L(x)=\8,
  \end{align}
  and consequently, for every $\e>0$
  \begin{align}\label{ratefi}
    x^{\g-\e}\lesssim_{\e}\vp(x),\ \ \ \mbox{and}\ \ \ \lim_{x\to\8}\frac{\vp(x)}{x}=0.
  \end{align}
  Furthermore, $x\mapsto x\vp(x)^{-\d}$ is increasing for every $\d< c$, (if $c=1$, even $\d\le1$ is allowed) and for every $x\ge h(x_0)$ we have
  \begin{align}\label{compfi}
    \vp(x)\simeq\vp(2x),\ \ \mbox{and}\ \ \vp'(x)\simeq\vp'(2x).
  \end{align}

\begin{lem}\label{lem1}
Assume that $c\in[1, 2)$ and $h\in\mathcal{F}_c$. If $c>1$, then for every $n\in\N$
\begin{align}\label{eq5}
  \lim_{x\to\8}\frac{x^nh^{(n)}(x)}{h(x)}=c(c-1)(c-2)\cdot\ldots\cdot(c-n+1).
\end{align}
If $c=1$, then $xh'(x)=h(x)(1+\vt(x))$ and for every $n\in\N$ with $n\ge2$
\begin{align}\label{eq6}
  \lim_{x\to\8}\frac{x^nh^{(n)}(x)}{\vt(x)h(x)}=(-1)^{n-2}(n-2)!.
\end{align}
\end{lem}
\begin{proof}
We may assume, without loss of generality, that the constant $C_h=1$. Firstly, we consider the case when $c>1$. Let $D_nf(x)=\frac{d^n}{dx}f(x)$ be the operator of the $n$--th derivative and $D_0f(x)=f(x)$. By the definition $h(x)=x^c\ell_h(x)$, then in view of  the Leibniz rule we have
\begin{align*}
  x^nD_nh(x)&=x^nD_n\big(x^c\ell_h(x)\big)=x^n\sum_{k=0}^n\binom{n}{k}D_k\big(x^c\big)D_{n-k}\ell_h(x)\\
  &=\sum_{k=0}^n\binom{n}{k}c_kh(x)\frac{x^{n-k}D_{n-k}\ell_h(x)}{\ell_h(x)},
\end{align*}
where $c_k=c(c-1)\cdot\ldots\cdot(c-k+1)$. The proof will be completed, if we show that
\begin{align}\label{eq7}
  \lim_{x\to\8}\frac{x^{k}D_{k}\ell_h(x)}{\ell_h(x)}=0, \ \ \ \mbox{for any $k\in\N$,}
\end{align}
then
\begin{align*}
  \lim_{x\to\8}\frac{x^nD_nh(x)}{h(x)}&=
  \lim_{x\to\8}\sum_{k=0}^n\binom{n}{k}c_k\frac{x^{n-k}D_{n-k}\ell_h(x)}{\ell_h(x)}=c_n,
\end{align*}
and we get \eqref{eq5}. In order to show \eqref{eq7} we will proceed by induction and prove that for every $k\in\N$ there exists a polynomial $P_k:\R^k\mapsto \R$ such that $P_k(0,\ldots,0)=0$ and
\begin{align*}
  x^{k}D_{k}\ell_h(x)=\ell_h(x)P_k\big(\vt(x), x\vt'(x), x^2\vt''(x),\ldots, x^{k-1}\vt^{(k-1)}(x)\big).
\end{align*}
Now we easily see that \eqref{eq7} follows by \eqref{eq2} and the continuity of polynomials, since $P_k(0,\ldots,0)=0$.
Our statement is true for $k=1$, since $x\ell'_h(x)=\ell_h(x)\vt(x)$. Now assume that it holds for all $1\le k\le n-1$. By induction we prove that it is true for $k=n$. For this purpose observe that $D_{n-1}(x\ell_h'(x))=D_{n-1}(\vt(x)\ell_h(x))$, thus
\begin{align*}
  xD_n\ell_h(x)+(n-1)D_{n-1}\ell_h(x)=\sum_{k=0}^{n-1}\binom{n-1}{k}D_k\vt(x)D_{n-1-k}\ell_h(x).
\end{align*}
Therefore, by the inductive hypothesis we obtain that
\begin{align*}
  x^nD_n\ell_h(x)&=\sum_{k=0}^{n-1}\binom{n-1}{k}x^kD_k\vt(x)\cdot x^{n-1-k}D_{n-1-k}\ell_h(x)-(n-1)x^{n-1}D_{n-1}\ell_h(x)\\
  &=\ell_h(x)\sum_{k=0}^{n-1}\binom{n-1}{k}x^kD_k\vt(x)\cdot P_{n-1-k}\big(\vt(x), x\vt'(x),\ldots, x^{n-1-k-1}\vt^{(n-1-k-1)}(x)\big)\\
  &-\ell_h(x)(n-1)P_{n-1}\big(\vt(x), x\vt'(x),\ldots, x^{n-2}\vt^{(n-2)}(x)\big).
\end{align*}
We have just shown that $x^{n}D_{n}\ell_h(x)=\ell_h(x)P_n\big(\vt(x), x\vt'(x),\ldots, x^{n-1}\vt^{(n-1)}(x)\big)$ for some polynomial $P_n:\R^n\mapsto\R$ such that $P(0,\ldots,0)=0$. This proves \eqref{eq7} and completes the proof for $c>1$.

Assume now that $c=1$ and we prove \eqref{eq6}. By the Leibniz formula we get
\begin{align*}
  D_nh(x)&=D_{n-1}\big(\ell_h(x)(1+\vt(x))\big)=D_{n-1}\ell_h(x)(1+\vt(x))
  +\ell_h(x)D_{n-1}\vt(x)\\
  &+(n-1)\ell_h'(x)D_{n-2}\vt(x)
  +\sum_{k=2}^{n-2}\binom{n-1}{k}D_k\ell_h(x)D_{n-1-k}\vt(x).
\end{align*}
We will show by induction with respect to $n\in\N$ that
\begin{align}\label{eq8}
  \lim_{x\to\8}\frac{x^nD_n\ell_h(x)}{\vt(x)\ell_h(x)}=(-1)^{n-1}(n-1)!,
\end{align}
and consequently \eqref{eq6} will follow, since by \eqref{eq8} and \eqref{eq4}, for $n\ge2$ we have
\begin{align*}
  \lim_{x\to\8}\frac{x^nD_nh(x)}{\vt(x)h(x)}&=\lim_{x\to\8}\frac{x^{n-1}D_{n-1}\ell_h(x)}{\vt(x)\ell_h(x)}(1+\vt(x))
  +\lim_{x\to\8}\frac{x^{n-1}D_{n-1}\vt(x)}{\vt(x)}\\
  &+(n-1)\lim_{x\to\8}\vt(x)\frac{x^{n-2}D_{n-2}\vt(x)}{\vt(x)}\\
  &+\lim_{x\to\8}\sum_{k=2}^{n-2}\binom{n-1}{k}\vt(x)
  \frac{x^kD_k\ell_h(x)}{\vt(x)\ell_h(x)}\frac{x^{n-1-k}D_{n-1-k}\vt(x)}{\vt(x)}\\
  &=\lim_{x\to\8}\frac{x^{n-1}D_{n-1}\ell_h(x)}{\vt(x)\ell_h(x)}=(-1)^{n-2}(n-2)!.
\end{align*}
For $n=1$ \eqref{eq8} is obviously satisfied, since $x\ell'_h(x)=\ell_h(x)\vt(x)$. Now assume that \eqref{eq8} is true for all $1\le k\le n-1$ we will show that it also remains true for $k=n$. Observe that the identity $D_{n-1}(x\ell_h'(x))=D_{n-1}(\vt(x)\ell_h(x))$, yields
\begin{align*}
  xD_n\ell_h(x)=(\vt(x)-(n-1))D_{n-1}\ell_h(x)+\sum_{k=1}^{n-1}\binom{n-1}{k}D_k\vt(x)D_{n-1-k}\ell_h(x),
\end{align*}
thus we obtain that
\begin{align*}
  \frac{x^nD_n\ell_h(x)}{\vt(x)\ell_h(x)}&=(\vt(x)-(n-1))\frac{x^{n-1}D_{n-1}\ell_h(x)}{\vt(x)\ell_h(x)}\\
  &+\sum_{k=1}^{n-1}\binom{n-1}{k}\frac{x^kD_k\vt(x)}{\vt(x)}\vt(x)\frac{x^{n-1-k}D_{n-1-k}\ell_h(x)}{\vt(x)\ell_h(x)},
\end{align*}
and consequently by the inductive hypothesis and \eqref{eq4} we get \eqref{eq8}. The proof of Lemma \ref{lem1} is completed.
\end{proof}
\begin{lem}\label{filem}
Assume that $c\in[1, 2)$ and $h\in\mathcal{F}_c$. Then for every $n\in\N$ there exists a function $\vt_n:[x_0, \8)\mapsto\R$ such that $\lim_{x\to\8}\vartheta_n(x)=0$ and
 \begin{align}\label{heq}
   xh^{(n)}(x)=h^{(n-1)}(x)(\a_n+\vt_n(x)),\ \ \mbox{for every \ $x\ge x_0$,}
 \end{align}
 where $\a_n=c-n+1$, $\vt_1(x)=\vt(x)$. If $c=1$ and $n=2$, then there exist constants $0<c_1\le c_2$ and a function $\vr:[x_0, \8)\mapsto[c_1, c_2]$, such that
 \begin{align}\label{heqc1}
   xh''(x)=h'(x)\vt(x)\vr(x),\ \ \mbox{for every \ $x\ge x_0$.}
 \end{align}
\end{lem}
\begin{proof}
We may assume, without loss of generality that the constant $C_h=1$. Since $h(x)=x^c\ell_h(x)$ and $x\ell'_h(x)=\ell_h(x)\vt(x)$, then
\begin{align*}
  h'(x)=x^{c-1}\ell_h(x)(c+\vt(x)),
\end{align*}
thus taking $\vt_1(x)=\vt(x)$ we obtain \eqref{heq} for $n=1$. Generally, we see that for $c>1$ and $n\ge1$, or $c=1$ and $n\ge3$ we have
\begin{align*}
  xh^{(n)}(x)=h^{(n-1)}(x)(\a_{n}+\vt_n(x)), \ \ \ \mbox{where} \ \ \ \vt_n(x)=\frac{xh^{(n)}(x)}{h^{(n-1)}(x)}-\a_{n},
\end{align*}
and $\lim_{x\to\8}\vt_n(x)=0$, by \eqref{eq5} or \eqref{eq6} respectively. If $c=1$ and $n=2$, then we have $xh''(x)=h'(x)\vt(x)\vr(x)$, where $\vr(x)=\frac{xh''(x)}{h'(x)\vt(x)}$ and $\lim_{x\to\8}\vr(x)=\lim_{x\to\8}\frac{x^2h''(x)}{\vt(x)h(x)}\frac{h(x)}{xh'(x)}=1$, by \eqref{eq6}.
This completes the proof of the lemma.
\end{proof}

We will look more closely at the function $\vp$ being the inverse function to the function $h\in\mathcal{F}_c$ and we collect all required properties its derivatives in the following.

\begin{lem}
Assume that $c\in[1, 2)$, $h\in\mathcal{F}_c$, $\g=1/c$ and let $\vp:[h(x_0), \8)\mapsto[x_0, \8)$ be its inverse. Let $\te$ be the function defined as in \eqref{tetadef}. If $c>1$, then for every $n\in\N$ we have
\begin{align}\label{eq9}
  \lim_{x\to\8}x^n\te^{(n)}(x)=0.
\end{align}
If $c=1$, then for every $n\in\N$ we have
\begin{align}\label{eq10}
  \lim_{x\to\8}\frac{x^n\te^{(n)}(x)}{\te(x)}=0.
\end{align}
\end{lem}
\begin{proof}
Let $D_nf(x)=\frac{d^n}{dx}f(x)$ denotes  the operator of the $n$--th derivative as in Lemma \ref{lem1} and $D_0f(x)=f(x)$. According to \eqref{tetadef} we know that $\te(h(x))=\frac{1}{(c+\vartheta(x))}-\g$. We only show the case when $c=1$, the same reasoning applies to the case $c>1$. Equivalently, it suffices to show that $\lim_{x\to\8}\frac{h(x)^nD_n\te(h(x))}{\te(h(x))}=0$. We shall proceed by induction with respect to $n\in\N$. For $n=1$, we note that
 \begin{align*}
   \lim_{x\to\8}\frac{h(x)D_1\te(h(x))}{\te(h(x))}=\lim_{x\to\8}\frac{ch(x)}{xh'(x)(c+\vt(x))}
   \cdot\frac{x\vt'(x)}{\vt(x)}=0,
 \end{align*}
 by \eqref{eq4} and Lemma \ref{lem1}. Assume now that \eqref{eq10} is true for any $1\le k\le n-1$. We show that it holds for $k=n$. First of all, observe that for every $n\in\N$, by the Fa\'{a} di Bruno formula applied to $[\te\circ h](x)$, we obtain
\begin{align}\label{eq12}
  D_{n}[\te\circ h](x)&=\sum_{\genfrac{}{}{0pt}{}{0\le m_1,\ldots,m_n\le n}{m_1+2m_2+\ldots+nm_n=n}}
  \binom{n}{m_1,\ldots, m_n}D_{m_1+\ldots+m_n}\te(h(x))\prod_{l=1}^n\bigg(\frac{D_lh(x)}{l!}\bigg)^{m_l}\\
 \nonumber =D_{n}\te(h(x))\big(D_1h(x)\big)^n&+
  \sum_{\genfrac{}{}{0pt}{}{0\le m_1,\ldots,m_n\le n, m_1\not=n}{m_1+2m_2+\ldots+nm_n=n}}
  \binom{n}{m_1,\ldots, m_n}D_{m_1+\ldots+m_n}\te(h(x))\prod_{l=1}^n\bigg(\frac{D_lh(x)}{l!}\bigg)^{m_l}\\
  \nonumber=D_{n}\te(h(x))\big(D_1h(x)\big)^n&+\sum_{r=1}^{n-1}D_{r}\te(h(x))
  \sum_{\tiny{\begin{array}{c}
  0\le m_1,\ldots,m_n\le n, m_1\not=n\\
  m_1+2m_2+\ldots+nm_n=n\\
  m_1+m_2+\ldots+m_n=r
  \end{array}
  }}a_{m_1,\ldots,m_n}\prod_{l=1}^n\big(D_lh(x)\big)^{m_l},
\end{align}
where $a_{m_1,\ldots,m_n}=\binom{n}{m_1,\ldots, m_n}\cdot\prod_{l=1}^n\left(\frac{1}{l!}\right)^{m_l}$.
By the Leibniz rule, we have
 \begin{align}\label{eq13}
  D_n\big((c+\vartheta(x))^2\big)=2(c+\vt(x))D_n\vt(x)+\sum_{k=1}^{n-1}\binom{n}{k}D_k\vt(x)D_{n-k}\vt(x).
 \end{align}
 Since $(c+\vartheta(x))^2[\te\circ h]'(x)=-\vt'(x)$ then one can see that again the Leibniz rule applied to $D_{n-1}\big((c+\vartheta(x))^2[\te\circ h]'(x)\big)$, gives
 \begin{align*}
   (c+\vartheta(x))^2D_{n}[\te\circ h](x)=-D_n\vt(x)-\sum_{k=1}^{n-1}\binom{n-1}{k}D_k\big((c+\vartheta(x))^2\big)D_{n-k}[\te\circ h](x).
 \end{align*}
 Thus combining the last identity with \eqref{eq12} and $\te(h(x))=\frac{-\vt(x)}{c(c+\vt(x))}$ we get
\begin{multline}\label{eq14}
   \frac{h(x)^nD_n\te(h(x))}{\te(h(x))}=\frac{c}{(c+\vartheta(x))}\frac{x^nD_n\vt(x)}{\vt(x)}
   \frac{h(x)^n}{\big(xD_1h(x)\big)^n}\\
   -
  \sum_{k=1}^{n-1}\binom{n-1}{k}
  \frac{h(x)^k}{\big(xD_1h(x)\big)^{k}}\frac{x^kD_k\big((c+\vartheta(x))^2\big)}{(c+\vartheta(x))^2}
  \frac{h(x)^{n-k}D_{n-k}[\te\circ h](x)}{\te(h(x))\big(D_1h(x)\big)^{n-k}}\\
  -\sum_{k=1}^{n-1}\frac{h(x)^kD_{k}\te(h(x))}{\te(h(x))}
  \sum_{\tiny{\begin{array}{c}
  0\le m_1,\ldots,m_n\le n, m_1\not=n\\
  m_1+2m_2+\ldots+nm_n=n\\
  m_1+m_2+\ldots+m_n=k
  \end{array}
  }}a_{m_1,\ldots,m_n}\prod_{l=1}^n\frac{h(x)^{n-k}\big(D_lh(x)\big)^{m_l}}{\big(D_1h(x)\big)^n}\\
  =I_1(x)+I_2(x)+I_3(x).
\end{multline}
Now we have to show that $\lim_{x\to\8}I_j(x)=0$ for $j=1, 2, 3$. By Lemma \ref{lem1} and \eqref{eq4} it is obvious that $\lim_{x\to\8}I_1(x)=0$. By Lemma \eqref{lem1} and \eqref{eq13} we have
\begin{align*}
  \lim_{x\to\8}\frac{h(x)^k}{\big(xD_1h(x)\big)^{k}}=c^k=1, \ \ \ \mbox{and}\ \ \ \ \lim_{x\to\8}\frac{x^kD_k\big((c+\vartheta(x))^2\big)}{(c+\vartheta(x))^2}=0,
\end{align*}
and $\lim_{x\to\8}I_2(x)=0$ if we show that for every $1\le k\le n-1$
\begin{align*}
 \lim_{x\to\8}\frac{h(x)^{k}D_{k}[\te\circ h](x)}{\te(h(x))\big(D_1h(x)\big)^{k}}=0.
\end{align*}
 For this purpose we use formula \eqref{eq12} and inductive hypothesis. Indeed,
 \begin{multline*}
\frac{h(x)^{k}D_{k}[\te\circ h](x)}{\te(h(x))\big(D_1h(x)\big)^{k}}=
\frac{h(x)^{k}D_{k}\te(h(x))}{\te(h(x))}\\
+\sum_{r=1}^{k-1}\frac{h(x)^{r}D_{r}\te(h(x))}{\te(h(x))}
  \sum_{\tiny{\begin{array}{c}
  0\le m_1,\ldots,m_k\le k, m_1\not=k\\
  m_1+2m_2+\ldots+km_k=k\\
  m_1+m_2+\ldots+m_k=r
  \end{array}
  }}a_{m_1,\ldots,m_k}\prod_{l=1}^k\frac{h(x)^{k-r}\big(D_lh(x)\big)^{m_l}}{\big(D_1h(x)\big)^{k}}.
\end{multline*}
In view of the inductive hypothesis it only remains to estimate the inner sum, or more precisely the last product. Namely, observe that
\begin{align*}
 \prod_{l=1}^k\frac{h(x)^{k-r}\big(D_lh(x)\big)^{m_l}}{\big(D_1h(x)\big)^{k}}
 &=\prod_{l=2}^k\frac{h(x)^{k-r}\big(D_lh(x)\big)^{m_l}}{\big(D_1h(x)\big)^{k-m_1}}=
 \frac{h(x)^{k-r}\big(\vt(x)h(x)\big)^{r-m_1}}{\big(xD_1h(x)\big)^{k-m_1}}
 \prod_{l=2}^k\left(\frac{x^lD_lh(x)}{\vt(x)h(x)}\right)^{m_l}\\
 &=\vt(x)^{r-m_1}\left(\frac{h(x)}{xD_1h(x)}\right)^{k-m_1}
 \prod_{l=2}^k\left(\frac{x^lD_lh(x)}{\vt(x)h(x)}\right)^{m_l},
\end{align*}
thus in view of Lemma \ref{lem1} and \eqref{eq4} the last expression tends to $0$ as $x\to\8$ as desired. In view of the inductive hypothesis one can show that $\lim_{x\to\8}I_3(x)=0$ arguing in a similar way as above. This finishes the proof of the lemma.
\end{proof}

\begin{lem}\label{funlemfi}
Assume that $c\in[1, 2)$, $h\in\mathcal{F}_c$, $\g=1/c$ and let $\vp:[h(x_0), \8)\mapsto[x_0, \8)$ be its inverse. Then
 for every $n\in\N$ there exists  a function $\theta_n:[h(x_0), \8)\mapsto\R$ such that $\lim_{x\to\8}\te_i(x)=0$ and
 \begin{align}\label{fiequat}
   x\vp^{(n)}(x)=\vp^{(n-1)}(x)(\b_n+\theta_n(x)), \ \ \mbox{for every \  $x\ge h(x_0)$,}
 \end{align}
  where $\b_n=\g-n+1$ and $\te_1(x)=\te(x)$. If $c=1$, then there exists a positive function $\s:[h(x_0), \8)\mapsto(0, \8)$ and a function $\t:[h(x_0), \8)\mapsto \R$ such that \eqref{fiequat} with $i=2$ reduces to
\begin{align}\label{fiequat1}
  x\vp''(x)=\vp'(x)\s(x)\t(x),\ \ \mbox{for every \  $x\ge h(x_0)$.}
\end{align}
Moreover, $\s(x)$  is decreasing, $\lim_{x\to\8}\s(x)=0,$ $\s(2x)\simeq\s(x),$ and  $\s(x)^{-1}\lesssim_{\varepsilon}x^{\varepsilon},$
for every $\varepsilon>0$. Finally, there are constants $0<c_3\le c_4$ such that  and $c_3\le-\t(x)\le c_4$ for every $x\ge h(x_0)$.
\end{lem}

\begin{proof}
In fact, \eqref{fiequat} for $n=1$ with $\te_1(x)=\te(x)$, has been shown in \eqref{tetadef}.
Arguing likewise in the proof of Lemma \ref{lem1} we show, in view of \eqref{eq9} and \eqref{eq10}, that if $c>1$, then for every $n\in\N$
\begin{align}\label{eq15}
  \lim_{x\to\8}\frac{x^n\vp^{(n)}(x)}{\vp(x)}=\g(\g-1)(\g-2)\cdot\ldots\cdot(\g-n+1).
\end{align}
If $c=1$, then for every $n\in\N$ with $n\ge2$
\begin{align}\label{eq16}
  \lim_{x\to\8}\frac{x^n\vp^{(n)}(x)}{\te(x)\vp(x)}=(-1)^{n-2}(n-2)!.
\end{align}
Now we see that \eqref{fiequat} and \eqref{fiequat1} follow from \eqref{eq15} and \eqref{eq16} and their proofs run as the proof of Lemma \ref{filem} with obvious modification for $c=1$ with $n=2$. In this case it is easy to see that it suffices to take
\begin{align*}
  \s(x)=1-\frac{1}{1+\vt(\vp(x))}=-\te(x), \ \ \ \mbox{and}\ \ \ \t(x)=-\frac{x\vp''(x)}{\vp'(x)\s(x)},
\end{align*}
and these have desired properties by \eqref{eq3} and Lemma \ref{filem}. It only remains to verify that $\s(2x)\simeq\s(x)$. For this purpose it is enough to prove that $\vt(2x)\simeq\vt(x)$.
Notice that for some $\xi_x\in(0, 1)$ we have
\begin{align*}
  \left|\frac{\vt(2x)}{\vt(x)}-1\right|=\left|\frac{(x+\xi_xx)\vt'(x+\xi_xx)}{\vt(x+\xi_xx)}\right|
  \frac{x}{x+\xi_xx}\frac{\vt(x+\xi_xx)}{\vt(x)}\le\left|\frac{(x+\xi_xx)\vt'(x+\xi_xx)}{\vt(x+\xi_xx)}\right|
  \ _{\overrightarrow{x\to\8}}\ 0,
\end{align*}
since $\vt(x)$ is decreasing. This completes the proof.
\end{proof}

\section{Estimates for some exponential sums}\label{sectexp}
In this section we will be concerned with the estimates of some exponential sums (see Lemma \ref{finboundlem} and Lemma \ref{vdclem2} below) which will be critical in the proof of Lemma \ref{formlem}. Let us recall that $\mu(n)$ denotes the M\"{o}bius function and $\Lambda(n)$ denotes von Mangoldt's function, i.e.
$$\mu(n)=\left\{ \begin{array} {ll}
\ \ 1, & \mbox{if $n=1$,}\\
(-1)^k, & \mbox{if $n$ is the product of $k$ distinct primes,}\\
\ \ 0,& \mbox{if $n$ is divisible by the square of a prime,}
\end{array}
\right.$$
and
$$\Lambda(n)=\left\{ \begin{array} {ll}
\log p, & \mbox{if $n=p^m$ for some $m\in\N$ and $p\in\mathbf{P}$,}\\
\ \ 0,& \mbox{otherwise.}
\end{array}
\right.$$
Our purpose will be to prove the following.
\begin{lem}\label{finboundlem}
Assume that $P\ge1$, $\xi\in[0, 1]$ and $M=P^{1+\chi+\e}\vp(P)^{-1}$ with  $\chi>0$  such that $(2^{2q+2}+2^q-2)(1-\g)+2^q(2^{q+3}-2)\chi<1$ and $0<\e<\frac{\chi}{100(2^{q+2}-1)}$. Let $W:\Z\mapsto\Z$ be a fixed polynomial of degree $q\in\N$.  Then for every $0< |m|\le M$ we have
\begin{multline}\label{finbound}
  \bigg|\sum_{P<k\le P_1\le 2P}\Lambda(k)e^{2\pi i(\xi W(k)+m\vp(k))}\bigg|\\
\lesssim |m|^{\frac{1}{2^{q+1}-2}}\log^2 P_1\ \big(\s(P_1)\vp(P_1)\big)^{-\frac{1}{2^q}}P_1^{1+\frac{2^{q+1}-2}{2^{2q+1}+2^q-2}}\\
 + |m|^{\frac{1}{2^{q+2}-2}}\log^{6}P_1\ \big(\s(P_1)\vp(P_1)\big)^{-\frac{2^{q+1}-2}{2^q(2^{q+2}-2)}}P_1^{1+\frac{2^q-1}{2^q(2^{q+2}-2)}}.
\end{multline}
If $c>1$ then the function $\s$ is constantly equal to $1$.
\end{lem}
The estimate \eqref{finbound} will be essential in the proof of Lemma \ref{formlem} in Section \ref{sectformlem}, there its need naturally arises. At the first glance it is difficult to deal with the exponential sum in \eqref{finbound} due to the occurrence of von Mangold function under the sum. However, exploring Vaughan's identity \eqref{vid}, we will be able to overcome this obstacle. If $v>n$ then
\begin{align}\label{vid}
  \Lambda(n)
  =\sum_{\genfrac{}{}{0pt}{}{kl=n}{l\le v}}\log k\ \mu(l)
  -\sum_{l\le v^2}\sum_{kl=n}\Pi_{v}(l)+\sum_{\genfrac{}{}{0pt}{}{kl=n}{k>v, l>v}}\Lambda(k)\Xi_v(l),
\end{align}
where
\begin{align}\label{pixi}
  \Pi_{v}(l)=\sum_{\genfrac{}{}{0pt}{}{rs=l}{r\le v, s\le v}}\Lambda(r)\mu(s), \ \ \ \mbox{and}\ \ \ \ \Xi_{v}(l)=\sum_{\genfrac{}{}{0pt}{}{d|l}{d>v}}\mu(d).
\end{align}
The proof of \eqref{vid} can be found in \cite{IK} see Proposition 13.4, page 345 or in \cite{GK} Lemma 4.12, page 49.
Moreover, for every $L\in\N$, we have
\begin{align}\label{aritmfun}
\sum_{L<l\le2L}|\Pi_{v}(l)|^2\lesssim L\log^2 L, \ \ \ \ \mbox{and}\ \ \ \ \sum_{L<l\le2L}|\Xi_v(l)|^2\lesssim L\log^3L.
\end{align}
\begin{proof}[Proof of Lemma \ref{finboundlem}]
Setting
\begin{align*}
  v=P_1^{\frac{2^{q+1}-2}{2^{2q+1}+2^q-2}},
\end{align*}
we immediately see, in view of \eqref{vid}, that
\begin{align}\label{vsplit}
 \sum_{P<n\le P_1\le 2P}\Lambda(n)&e^{2\pi i(\xi W(n)+m\vp(n))}=
 \sum_{l\le v}\sum_{P/l<k\le P_1/l}\log k\ \mu(l)e^{2\pi i(\xi W(kl)+m\vp(kl))}\\
\nonumber &-\bigg(\sum_{l\le v}+\sum_{v<l\le v^2}\bigg)\sum_{P/l<k\le P_1/l}\Pi_{v}(l)e^{2\pi i(\xi W(kl)+m\vp(kl))}\\
 \nonumber &+\sum_{v<l\le P_1/v}\sum_{\genfrac{}{}{0pt}{}{P/l< k\le P_1/l}{k>v}}\Lambda(k)\Xi_v(l)e^{2\pi i(\xi W(kl)+m\vp(kl))}=S_1-S_{21}-S_{22}+S_3.
\end{align}
Hence we are reduced to estimate these sums. It suffices to show that
\begin{align}\label{ineq1}
  |S_1|,\ |S_{21}|\lesssim |m|^{\frac{1}{2^{q+1}-2}}\log^2 P_1\ \big(\s(P_1)\vp(P_1)\big)^{-\frac{1}{2^q}}P_1^{1+\frac{2^{q+1}-2}{2^{2q+1}+2^q-2}},
\end{align}
and
\begin{align}\label{ineq2}
 |S_{22}|,\ |S_{3}|\lesssim |m|^{\frac{1}{2^{q+2}-2}}\log^{6}P_1\ \big(\s(P_1)\vp(P_1)\big)^{-\frac{2^{q+1}-2}{2^q(2^{q+2}-2)}}P_1^{1+\frac{2^q-1}{2^q(2^{q+2}-2)}}.
\end{align}
The bounds in \eqref{ineq1} and \eqref{ineq2} will be proved in the next two subsections.
\end{proof}
The proof of the inequalities \eqref{ineq1} and \eqref{ineq2} to a large extent will be based on the Van der Corput estimates.
\begin{lem}[Van der Corput \cite{VDC}, \cite{GK} or \cite{IK}]\label{vdc}
Suppose that $N\ge1$, $k\geq2$ is an integer and $a\le b \le a+N$. Let $F\in\mathcal{C}^k([a, b])$ be a real valued function such that
\begin{align*}
  \eta\lesssim |F^{(k)}(x)|\lesssim r\eta,\ \ \mbox{for every \ $x\in [a, b]$,}
\end{align*}
for some $\eta>0$ and $r\ge1$. Then
\begin{align*}
    \bigg|\sum_{a\le n\le b}\ep{F(n)}\bigg|\lesssim
    rN\left(\eta^{1/(2^k-2)}+N^{-2/2^k}+(N^k\eta)^{-2/2^k}\right),
\end{align*}
where the implied constant is absolute.
\end{lem}
With the aid of Lemma \ref{vdc} we will derive a very useful estimate in the next lemma.
\begin{lem}\label{vdclem2}
Let $W:\Z\mapsto\Z$ be a fixed polynomial of degree $q\in\N$. For every $m\in\Z\setminus\{0\}$, $l\in\N$, $j\ge0$ and $X\ge 1$ we have
\begin{align}\label{vdcest3}
  \bigg|\sum_{1\le k\le X}\ e^{2\pi i(\xi jW(kl)+m\vp(kl))}\bigg|\lesssim |m|^{1/(2^{q+1}-2)}\log(lX)\ lX\big(\s(lX)\vp(lX)\big)^{-1/2^q}.
\end{align}
If $c>1$ then $\s$ is constantly equal to $1$.
\end{lem}

\begin{proof}
Let $U_{j, l}(X)$ denote the sum in \eqref{vdcest3}, and  split  $U_{j, l}(X)$ into $\log X$ dyadic pieces of the form $\sum_{Y<k\le Y'\le 2Y}e^{2\pi i(\xi jW(kl)+m\vp(kl))}$, where $Y\in[1, X]$. One can assume, that $m>0$ and let $F(t)=\xi jW(lt)+m\vp(lt)$ for $t\in[Y, 2Y]$. According to Lemma \ref{funlemfi} we know that $\frac{x^n\vp^{(n)}(x)}{\s(x)\vp(x)}\simeq C_{\vp, n}\not=0$ for every $n\in\N$, (if $c>1$ one can think that $\s$ is constantly equal to $1$). Thus
$$|F^{(q+1)}(t)|=|ml^{q+1}\vp^{(q+1)}(lt)|\simeq \frac{ml^{q+1}\s(lY)\vp(lY)}{(lY)^{q+1}},$$
and consequently by Lemma \ref{vdc}
we obtain
\begin{multline*}
 \bigg|\sum_{Y<k\le Y'\le 2Y}e^{2\pi i(\xi jW(kl)+m\vp(kl))}\bigg| \\
\lesssim Y\left( \left(\frac{ml^{q+1}\s(lY)\vp(lY)}{(lY)^{q+1}}\right)^{1/(2^{q+1}-2)}+
Y^{-1/2^q}+\left(Y^{q+1}\frac{ml^{q+1}\vp(lY)\s(lY)}{(lY)^{q+1}}\right)^{-1/2^q}\right)\\
\lesssim (ml)^{1/(2^{q+1}-2)}Y^{1-q/(2^{q+1}-2)}+
Y^{1-1/2^q}+Y\left(\frac{1}{\s(lY)\vp(lY)}\right)^{1/2^q}\\
 \lesssim m^{1/(2^{q+1}-2)}lY\big(\s(lY)\vp(lY)\big)^{-1/2^q}.
\end{multline*}
Finally we obtain that
\begin{align*}
   |U_{j, l}(X)|&\lesssim\log X\sup_{Y\in[1, X]}m^{1/(2^{q+1}-2)}lY\big(\s(lY)\vp(lY)\big)^{-1/2^q}\\
&\lesssim m^{1/(2^{q+1}-2)}\log(lX)lX\big(\s(lX)\vp(lX)\big)^{-1/2^q},
 \end{align*}
 since $x\mapsto x\big(\s(x)\vp(x)\big)^{-1/2^q}$ is increasing. The proof of Lemma \ref{vdclem2} follows.
\end{proof}

In the sequel, we will use the following version  of summation by parts.
\begin{lem}[see \cite{Nat}]\label{sbp}
Let $0\le a<b$ be real numbers and $g(n)$ and $u(n)$   be arithmetic functions such that $g\in\mathcal{C}^1([a, b])$ and $U(t)=\sum_{a< n\le t}u(n)$. Then
$$\sum_{a<n\le b}u(n)g(n)=U(b)g(b)-\int_a^bU(t)g'(t)dt.$$
\end{lem}

\subsection{The estimates for $S_1$ and $S_{21}$}
Let us define $U_l(x)=\sum_{P/l< k\le x}\ e^{2\pi i(\xi W(lk)+m\vp(lk))}$ and recall that $v=P_1^{\frac{2^{q+1}-2}{2^{2q+1}+2^q-2}}$. Applying summation by parts to the inner sum in $S_1$ we see that
\begin{align*}
  S_1=\sum_{l\le v}\mu(l)\sum_{P/l<k\le P_1/l}\log k e^{2\pi i(\xi W(kl)+m\vp(kl))}
  =\sum_{l\le v}\mu(l)\bigg(U_l(P_1/l)\log(P_1/l)-\int_{P/l}^{P_1/l}U_l(x)\frac{dx}{x}\bigg).
\end{align*}
This implies
\begin{align*}
  |S_1|\le\log P_1\ \sum_{l\le v}\sup_{P/l\le x\le P_1/l}|U_l(x)|.
\end{align*}
Moreover,
\begin{align*}
  |S_{21}|\le\sum_{l\le v}|\Pi_v(l)||U_l(P_1/l)|
 \lesssim \log P_1\ \sum_{l\le v}|U_l(P_1/l)|,
\end{align*}
since $|\Pi_v(l)|\le\sum_{k|l}\Lambda(k)=\log l$. Thus Lemma \ref{vdclem2} applied to $U_l(x)$ implies that
\begin{align*}
  |S_1|,\ |S_{21}| &\le\log P_1\ \sum_{l\le v}\sup_{P/l\le x\le P_1/l}|U_l(x)|\\
&\lesssim
\log P_1\ \sum_{l\le v}\sup_{P/l\le x\le P_1/l}|m|^{\frac{1}{2^{q+1}-2}}\log(lx)lx\big(\s(lx)\vp(lx)\big)^{-\frac{1}{2^q}}\\
&\lesssim
|m|^{\frac{1}{2^{q+1}-2}}\log^2 P_1\ \big(\s(P_1)\vp(P_1)\big)^{-\frac{1}{2^q}}P_1^{1+\frac{2^{q+1}-2}{2^{2q+1}+2^q-2}},
\end{align*}
since $x\mapsto x\big(\s(x)\vp(x)\big)^{-1/2^q}$ is increasing and the proof of \eqref{ineq1} is completed.
\subsection{The estimates for $S_{22}$ and $S_3$}
Here we shall bound $S_{22}$ and $S_3$.  For $S_{22}$, we have
\begin{align}\label{s22}
  |S_{22}|&=\bigg|\sum_{v<l\le v^2}\sum_{P/l<k\le P_1/l}\Pi_{v}(l)e^{2\pi i(\xi W(kl)+m\vp(kl))}\bigg|\\
  \nonumber&\lesssim\log^2P_1\sup_{L\in[v, v^2]}\sup_{K\in[P/v^2, P_1/v]}\bigg|\sum_{L<l\le L'\le 2L}
  \sum_{\genfrac{}{}{0pt}{}{K<k\le K'\le 2K}{P<kl\le P_1}}\Pi_{v}(l)e^{2\pi i(\xi W(kl)+m\vp(kl))}\bigg|,
\end{align}
and for $S_3$, we have
\begin{align}\label{s3}
  |S_3|&=\bigg|\sum_{v<l\le P_1/v}\sum_{\genfrac{}{}{0pt}{}{P/l< k\le P_1/l}{k>v}}\Lambda(k)\Xi_v(l)e^{2\pi i(\xi W(kl)+m\vp(kl))}\bigg|\\
 \nonumber &\lesssim\log^2P_1\sup_{L\in[v, P_1/v]}\sup_{K\in[v, P_1/v]}\bigg|
  \sum_{L<l\le L'\le 2L}\sum_{\genfrac{}{}{0pt}{}{K<k\le K'\le2K}{P< kl\le P_1}}\Lambda(k)\Xi_v(l)e^{2\pi i(\xi W(kl)+m\vp(kl))}\bigg|.
\end{align}
In view of these decompositions it remains to show.
\begin{lem}\label{billem}
Let $K, L\in\N$, $m\in\Z\setminus\{0\}$. Assume that $\vp(KL)\le \min\{K, L\}^{\frac{2^{2q+1}+2^q-2}{2^{q+1}-2}}$ and $|m|\min\{K, L\}^{\frac{2^{q+1}-2}{2^q}}\le \big(\s(KL)\vp(KL)\big)^{\frac{2^{q+1}-2}{2^q}}$. Then
 \begin{align}\label{billem1}
  \bigg|\sum_{L<l\le L'\le 2L}\sum_{\genfrac{}{}{0pt}{}{K<k\le K'\le2K}{P< kl\le P_1}}\Delta_1(l)\Delta_2(k)&e^{2\pi i(\xi W(kl)+m\vp(kl))}\bigg|\\
\nonumber\lesssim |m|^{\frac{1}{2^{q+2}-2}}\ \log^{2}L\ \log^{2}K\ &\big(\s(KL)\vp(KL)\big)^{-\frac{2^{q+1}-2}{2^q(2^{q+2}-2)}}
\ \min\{K, L\}^{\frac{2^{q+1}-2}{2^q(2^{q+2}-2)}}\ KL,
\end{align}
for every sequences of complex numbers $(\Delta_1(l))_{l\in(L, 2L]}$, and  $(\Delta_2(k))_{k\in(K, 2K]}$ such that
\begin{align}\label{logineq}
 \sum_{L<l\le 2L}|\Delta_1(l)|^2\lesssim L\log^{3}L,\ \ \ \mbox{and}\ \ \ \sum_{K<k\le 2K}|\Delta_2(k)|^2\lesssim K\log^{3}K.
\end{align}
\end{lem}
Assuming Lemma \ref{billem} we would have the bounds for $S_{22}$ and $S_3$ as in \eqref{ineq2}. Indeed, recall that $M=P^{1+\chi+\e}\vp(P)^{-1}$ with $\chi>0$  such that $(2^{2q+2}+2^q-2)(1-\g)+2^q(2^{q+3}-2)\chi<1$ and $0<\e<\frac{\chi}{100(2^{q+2}-1)}$.  Observe that
$$v=P_1^{\frac{2^{q+1}-2}{2^{2q+1}+2^q-2}}\le P_1^{1-\frac{2^{q+2}-4}{2^{2q+1}+2^q-2}}=P_1/v^2,$$
and consequently $v\le P_1^{1/2}$ and $v^2\le P_1/v$.
Thus in both cases $K, L\in[v, P_1/v]$.
Therefore, $v\le \min\{K, L\}\le P_1^{1/2}$ since $KL\simeq P_1$, and for sufficiently large $P_1\simeq P$, we see that $\vp(KL)\le \min\{K, L\}^{\frac{2^{2q+1}+2^q-2}{2^{q+1}-2}}$. If not, then $$\min\{K, L\}^{\frac{2^{2q+1}+2^q-2}{2^{q+1}-2}}<\vp(KL)\le \vp(P_1)\le P_1,$$ hence $\min\{K, L\}< P_1^{\frac{2^{q+1}-2}{2^{2q+1}+2^q-2}}=v$ contrary to what we have just shown. Finally, it remains to verify that $|m|\min\{K, L\}^{\frac{2^{q+1}-2}{2^q}}\le \big(\s(KL)\vp(KL)\big)^{\frac{2^{q+1}-2}{2^q}}$. Indeed, $2-1/2^q+\chi+4\e -\g(3-2/2^q)\le1/2^q((3\cdot2^q-2)(1-\g)-2^q+1+5\cdot2^q\chi)\le 1/2^q((2^{2q+2}+2^q-2)(1-\g)+2^q(2^{q+3}-2)\chi-1)<0$, thus \begin{multline*}
 |m|\min\{K, L\}^{\frac{2^{q+1}-2}{2^q}}\le MP_1^{1-1/2^q}\\
 =P_1^{2-1/2^q+\chi+\e}\vp(P_1)^{-1}\big(\s(KL)\vp(KL)\big)^{-(2-2/2^q)}
\big(\s(KL)\vp(KL)\big)^{\frac{2^{q+1}-2}{2^q}}\\
\lesssim P_1^{2-1/2^q+\chi+4\e -\g(3-2/2^q)}\cdot\big(\s(KL)\vp(KL)\big)^{\frac{2^{q+1}-2}{2^q}}\le \big(\s(KL)\vp(KL)\big)^{\frac{2^{q+1}-2}{2^q}}.
\end{multline*}
Therefore, \eqref{billem1} yields
\begin{multline*}
\bigg|\sum_{L<l\le L'\le 2L}\sum_{\genfrac{}{}{0pt}{}{K<k\le K'\le2K}{P< kl\le P_1}}\Delta_1(l)\Delta_2(k)e^{2\pi i(\xi W(kl)+m\vp(kl))}\bigg|\\
  \lesssim |m|^{\frac{1}{2^{q+2}-2}}\ \log^{2}L\ \log^{2}K\ \big(\s(KL)\vp(KL)\big)^{-\frac{2^{q+1}-2}{2^q(2^{q+2}-2)}}
\ \min\{K, L\}^{\frac{2^{q+1}-2}{2^q(2^{q+2}-2)}}\ KL\\
\lesssim |m|^{\frac{1}{2^{q+2}-2}}\ \log^{4}P_1\ \big(P_1^{1/2}\big)^{\frac{2^{q+1}-2}{2^q(2^{q+2}-2)}}\ P_1\ \big(\s(P_1)\vp(P_1)\big)^{-\frac{2^{q+1}-2}{2^q(2^{q+2}-2)}}\\
\lesssim |m|^{\frac{1}{2^{q+2}-2}}\log^{4}P_1\ \big(\s(P_1)\vp(P_1)\big)^{-\frac{2^{q+1}-2}{2^q(2^{q+2}-2)}}P_1^{1+\frac{2^q-1}{2^q(2^{q+2}-2)}}.
   \end{multline*}
This in turn completes the proof of the estimates \eqref{ineq2}, since after appropriate choice of $\Delta_1(l)$ and $\Delta_2(k)$ in $S_{22}$ and $S_3$, \eqref{aritmfun} shows that \eqref{logineq} is satisfied, we get
\begin{align*}
  |S_{22}|,\ |S_3|\lesssim |m|^{\frac{1}{2^{q+2}-2}}\log^{6}P_1\ \big(\s(P_1)\vp(P_1)\big)^{-\frac{2^{q+1}-2}{2^q(2^{q+2}-2)}}P_1^{1+\frac{2^q-1}{2^q(2^{q+2}-2)}},
\end{align*}
where the additional $\log^2P_1$ factor comes form the dyadic decompositions \eqref{s22} and \eqref{s3}.

\begin{proof}[Proof of Lemma \ref{billem}]
In view of the symmetry between the variables $k, l$ in the sums in \eqref{billem1} one can assume that $K\le L$.
We will divide the proof into three steps and we are going to follow the concepts of Heath--Brown from \cite{HB} Section 5, see also \cite{GK} Section 4.\\

\noindent {\textsf{\textbf{\underline{Step 1.}}}} Let us define
 \begin{align*}
   E_r=\sum_{L<l\le2L}\sum_{\genfrac{}{}{0pt}{}{K<k, k+r\le K'\le 2K}{P< kl, (k+r)l\le P_1}}\Delta_2(k)\overline{\Delta_2(k+r)}e^{2\pi i (\xi W(kl)+m\vp(kl) - \xi W((k+r)l)-m\vp((k+r)l))},
 \end{align*}
 for every $r\in\Z$. If $r=0$ we see, by \eqref{logineq}, that
 \begin{align}\label{eineq0}
   |E_0|\le\sum_{L<l\le2L}\sum_{K<k\le K'\le2K}|\Delta_2(k)|^2\lesssim L\sum_{K/2<k\le 2K}|\Delta_2(k)|^2\lesssim LK\log^3K.
 \end{align}
 Moreover, setting
 \begin{align*}
  \widetilde{S}(k, r)=\sum_{\max\{L, \frac{P}{k}, \frac{P}{k+r}\}<l\le\min\{2L, \frac{P_1}{k}, \frac{P_1}{k+r}\}}e^{2\pi i (\xi W(kl)+m\vp(kl) - \xi W((k+r)l)-m\vp((k+r)l))},
\end{align*}
 we see that for any $r\in\Z\setminus\{0\}$ we have
  \begin{align*}
   E_r=\sum_{\max\{K, K-r\}<k\le \min\{K', K'-r\}}\Delta_2(k)\overline{\Delta_2(k+r)}\widetilde{S}(k, r).
 \end{align*}
One can see that for every $R\ge1$ we have
\begin{multline}\label{eineq}
   \sum_{1\le|r|\le R}|E_r|\lesssim \sum_{1\le|r|\le R}\sum_{K<k, k+r\le K'}|\Delta_2(k)|^2|\widetilde{S}(k, r)|+|\overline{\Delta_2(k+r)}|^2|\widetilde{S}(k+r, -r)|\\
  \le\sum_{1\le|r|\le R}\sum_{K<k, k+r\le K'}|\Delta_2(k)|^2|\widetilde{S}(k, r)|
   +\sum_{1\le|r|\le R}\sum_{K<k, k-r\le K'}|\Delta_2(k)|^2|\widetilde{S}(k, -r)|\\
   \lesssim\sum_{1\le|r|\le R}\sum_{K<k, k+r\le K'}|\Delta_2(k)|^2|\widetilde{S}(k, r)|
   =\sum_{K<k\le K'}|\Delta_2(k)|^2\sum_{1\le|r|\le R}|\widetilde{S}(k, r)|
   \mathbf{1}_{(K, K']}(k+r),
 \end{multline}
since $|\widetilde{S}(k, r)|=|\widetilde{S}(k+r, -r)|$.\\

 \noindent {\textsf{\textbf{\underline{Step 2.}}}} We will prove that for every $m\in\N$, $k\in(K, 2K]$ and $R\ge1$ we have
  \begin{align}\label{vdcest4}
    \frac{1}{R}\sum_{1\le |r|<R}|\widetilde{S}(k, r)|\mathbf{1}_{(K, 2K]}(k+r)\lesssim (mR)^{1/(2^{q+1}-2)}KL\big(\s(KL)\vp(KL)\big)^{-1/2^q}K^{1/2^q-1}.
  \end{align}
For this purpose define $F(x)=\xi W(kx)+m\vp(kx) - \xi W((k+r)x)-m\vp((k+r)x))$, for $x\in(L, 2L]$. Then, according to Lemma \ref{funlemfi} and the mean value theorem,  for some $\eta\in(0, 1)$ and $\eta_{k, r}=k+\eta r$ if $r>0$ and $\eta_{k, r}=k+r-\eta r$ if $r<0$, we have
\begin{multline*}
 |F^{(q+1)}(x)|=|mk^{q+1}\vp^{(q+1)}(kx)-m(k+r)^{q+1}\vp^{(q+1)}((k+r)x))|\\
= \big|r\big((q+1)m\eta_{k, r}^q\vp^{(q+1)}(x\eta_{k, r})+m\eta_{k, r}^{q+1}x\vp^{(q+2)}(x\eta_{k, r})\big)\big|\\
= |rm\eta_{k, r}^q\vp^{(q+1)}(x\eta_{k, r})(q+1+\b_{q+2}+\theta_{q+2}(x\eta_{k, r}))|\\
\simeq |mrK^q\vp^{(q+1)}(KL)|\simeq \frac{m|r|K^q\s(KL)\vp(KL)}{(KL)^{q+1}},
\end{multline*}
since $k, k+r\in(K, 2K]$ and $\eta_{k, r}\in(K, 2K]$. Therefore, by Lemma \ref{vdc} we obtain
\begin{multline*}
  |\widetilde{S}(k, r)|
  \lesssim L\left(\left(\frac{m|r|K^q\s(KL)\vp(KL)}{(KL)^{q+1}}\right)^{1/(2^{q+1}-2)}+L^{-1/2^q}+
\left(L^{q+1}\frac{m|r|K^q\s(KL)\vp(KL)}{(KL)^{q+1}}\right)^{-1/2^q}\right)\\
  \lesssim (m|r|)^{1/(2^{q+1}-2)}L^{1-q/(2^{q+1}-2)}+L^{1-1/2^q}
+LK^{1/2^q}\left(\frac{1}{\s(KL)\vp(KL)}\right)^{1/2^q}\\
\lesssim (m|r|)^{1/(2^{q+1}-2)}L^{1-1/2^q}+KL\left(\frac{1}{\s(KL)\vp(KL)}\right)^{1/2^q}K^{1/2^q-1}\\
  \lesssim (m|r|)^{1/(2^{q+1}-2)}KL\big(\s(KL)\vp(KL)\big)^{-1/2^q}K^{1/2^q-1},
\end{multline*}
and \eqref{vdcest4} follows (if $c>1$, as before, one can think that $\s$ is constantly equal to $1$). Combining \eqref{eineq} with \eqref{vdcest4} we obtain that
\begin{align}\label{eineq1}
  \frac{1}{R}\sum_{1\le|r|\le R}|E_r|&\lesssim\sum_{K<k\le K'}|\Delta_2(k)|^2\frac{1}{R}\sum_{1\le|r|\le R}|\widetilde{S}(k, r)|\mathbf{1}_{(K, K']}(k+r)\\
  \nonumber&\lesssim K\log^3K\cdot (mR)^{1/(2^{q+1}-2)}KL\big(\s(KL)\vp(KL)\big)^{-1/2^q}K^{1/2^q-1}.
\end{align}

\noindent {\textsf{\textbf{\underline{Step 3.}}}} Now we can finish our proof. We shall apply Weyl--Van der Corput shift inequality (see \cite{HB} Lemma 5, page 258), which asserts that for a fixed $U\ge1$, any complex number $z_u\in\C$ with $U<u\le 2U$ and any interval $I\subseteq(U, 2U]$ we have for every $R\in\N$ that
$$\bigg|\sum_{u\in I}z_u\bigg|^2\le \frac{U+R}{R}\sum_{|r|<R}\left(1-\frac{|r|}{R}\right)\sum_{u, u+r\in I}z_u\overline{z}_{u+r}.$$
By the Cauchy--Schwartz inequality and Weyl--Van der Corput shift inequality, applied with $U=K$ and an integer $1\le R\lesssim K$ which will be adjusted later, we see that

\begin{multline}\label{eineq2}
   \bigg|\sum_{L<l\le2L}\sum_{\genfrac{}{}{0pt}{}{K<k\le K'\le 2K}{P< kl\le P_1}}\Delta_1(l)\Delta_2(k)e^{2\pi i (\xi W(kl)+m\vp(kl)}\bigg|^2\\
   \le\bigg(\sum_{L<l\le2L}|\Delta_1(l)|^2\bigg) \sum_{L<l\le L'\le 2L}\bigg|\sum_{\genfrac{}{}{0pt}{}{K<k \le K'\le 2K}{P< kl\le P_1}}\Delta_2(k)e^{2\pi i (\xi W(kl)+m\vp(kl)}\bigg|^2\\
   \lesssim L\log^{3}L \sum_{L<l\le2L}\bigg|\sum_{\genfrac{}{}{0pt}{}{K<k\le K'\le 2K}{P< kl\le P_1}}\Delta_2(k)e^{2\pi i (\xi W(kl)+m\vp(kl)}\bigg|^2\\
   \lesssim L\log^{3}L\ \frac{K+R}{R}\sum_{|r|\le R}\left(1-\frac{|r|}{R}\right)|E_r|\\
   \lesssim L^2K\log^{3}L\log^{3}K\ \frac{K+R}{R}+L\log^{3}L\ \frac{K+R}{R}\sum_{1\le|r|\le R}|E_r|\\
   \lesssim \log^{3}L\log^{3}K\left(\frac{L^2K^2}{R}+K^2L (mR)^{1/(2^{q+1}-2)}KL\big(\s(KL)\vp(KL)\big)^{-1/2^q}K^{1/2^q-1}\right).
 \end{multline}
We have used the estimate \eqref{eineq0} for $|E_0|$ and the inequality \eqref{eineq1}. Now let us define $R=\big\lceil m^{-a}K^{-b}L^c\big(\s(KL)\vp(KL)\big)^{-d}\big\rceil$ for some $a, b, c, d\in\R$ and oberve that the last expression in \eqref{eineq2} is bounded by
\begin{multline*}
\log^{3}L\log^{3}K\Big(m^aK^{2+b}L^{2-c}\big(\s(KL)\vp(KL)\big)^{d}\\
+ m^{\frac{1-a}{2^{q+1}-2}}K^{2+\frac{1}{2^q}-\frac{b}{2^{q+1}-2}}
L^{2+\frac{c}{2^{q+1}-2}}\big(\s(KL)\vp(KL)\big)^{-\frac{d}{2^{q+1}-2}-\frac{1}{2^q}}\Big).
   \end{multline*}
It suffices to arrange the parameters $a, b, c, d\in\R$ so that to make the last two terms equal. Namely, it is enough to take
\begin{align*}
  a=\frac{1-a}{2^{q+1}-2}&\Longleftrightarrow a=\frac{1}{2^{q+1}-1},\\
2+b=2+\frac{1}{2^q}-\frac{b}{2^{q+1}-2}&\Longleftrightarrow b=\frac{1}{2^q}\frac{2^{q+1}-2}{2^{q+1}-1},\\
2-c=2+\frac{c}{2^{q+1}-2}&\Longleftrightarrow c=0,\\
d=-\frac{d}{2^{q+1}-2}-\frac{1}{2^q}&\Longleftrightarrow d=-\frac{1}{2^q}\frac{2^{q+1}-2}{2^{q+1}-1}.
\end{align*}
We now easily see, since we have assumed that $K\le L$, that $$1\le m^{-\frac{1}{2^{q+1}-1}}K^{-\frac{2^{q+1}-2}{2^q(2^{q+1}-1)}}\big(\s(KL)\vp(KL)\big)^{\frac{2^{q+1}-2}{2^q(2^{q+1}-1)}}\le K,$$ by our assumptions, thus $1\le R\lesssim K$ and consequently \eqref{billem1} follows, since
\begin{multline*}
\bigg|\sum_{L<l\le2L}\sum_{\genfrac{}{}{0pt}{}{K<k\le K'\le 2K}{P< kl\le P_1}}\Delta_1(l)\Delta_2(k)e^{2\pi i (\xi W(kl)+m\vp(kl)}\bigg|\\
\lesssim  m^{\frac{1}{2^{q+2}-2}}\ \log^{2}L\ \log^{2}K\ \big(\s(KL)\vp(KL)\big)^{-\frac{2^{q+1}-2}{2^q(2^{q+2}-2)}}
\ K^{\frac{2^{q+1}-2}{2^q(2^{q+2}-2)}}\ KL.
   \end{multline*}
   This completes the proof of Lemma \ref{billem}.
   \end{proof}
\section{The main lemma}\label{sectformlem}
We have just gathered all necessary estimates for the exponential sums in Section \ref{sectexp} and now we can formulate the main lemma of this paper. Lemma \ref{formlem} is the hearth of the matter and will allow us to prove both Theorem \ref{maxithm} and Theorem \ref{asymptthm}. Recall that $c_q=(2^{2q+2}+2^q-2)/(2^{2q+2}+2^q-3)$.
\begin{lem}\label{formlem}
Let $W:\Z\mapsto\Z$ be a fixed polynomial of degree $q\in\N$. Assume that $h\in\mathcal{F}_c$, $\vp$ be its inverse and $\g=1/c$ with $c\in[1, c_q)$. Let $\chi>0$ be a number obeying $(2^{2q+2}+2^q-2)(1-\g)+2^q(2^{q+3}-2)\chi<1$, then there exists $\chi'>0$ such that for every $N\in\N$ and for every $\xi\in[0, 1]$
\begin{align}\label{form}
   \sum_{p\in\mathbf{P}_{h, N}}\vp'(p)^{-1}\log p\ e^{2\pi i \xi W(p)}=\sum_{p\in\mathbf{P}_{N}} \log p\ e^{2\pi i \xi W(p)}+O\big(N^{1-\chi-\chi'}\big).
\end{align}
The implied constant is independent of $\xi\in[0, 1]$ and $N\in\N$.
\end{lem}

We shall provide detailed proof of Lemma \ref{formlem}, which will be based on the ideas of Heath--Brown  from \cite{HB}.
A variant of Lemma \ref{formlem}, for the set of Piatetski--Shapiro primes with $W(x)=x$, was proved by Balog and Friedlander \cite{BF} and by Kumchev \cite{Kum}. They used this result to show that the ternary Goldbach problem has a solution in the Piatetski--Shapiro primes. We start with the following.
\begin{lem}\label{lemest0}
Let $\Phi(x)=\{x\}-1/2$ and $\Lambda(n)$ denotes von Mangoldt's. Then, under the assumptions of Lemma \ref{formlem}, for every $N\in\N$ and $0<\e<\frac{\chi}{100(2^{q+2}-1)}$ we have
\begin{align}\label{lem0eq1}
  \sum_{p\in\mathbf{P}_{h, N}}\vp'(p)^{-1}\log p\ e^{2\pi i \xi W(p)}&=\sum_{p\in\mathbf{P}_{N}} \log p\ e^{2\pi i \xi W(p)}\\
 \nonumber +\sum_{k=1}^N \vp'(k)^{-1}&\big(\Phi(-\vp(k+1))-\Phi(-\vp(k))\big)\Lambda(k)e^{2\pi i \xi W(k)}+
  O\big(N^{1-\chi-\e}\big).
\end{align}
\end{lem}
\begin{proof}
Form \cite{M1} we know that $p\in\mathbf{P}_{h} \Longleftrightarrow\ \lfloor-\vp(p)\rfloor-\lfloor-\vp(p+1)\rfloor=1,$
for all sufficiently large $p\in\mathbf{P}_{h}$. By the definition of function $\Phi(x)=\{x\}-1/2$ we obtain that for every $p\in\N$ there exists  $\xi_p\in(0, 1)$ such that
\begin{align*}
  \lfloor-\vp(p)\rfloor-\lfloor-\vp(p+1)\rfloor&=\vp(p+1)-\vp(p)+\Phi(-\vp(p+1))-\Phi(-\vp(p))\\
  &=\vp'(p)+\vp''(p+\xi_p)/2+\Phi(-\vp(p+1))-\Phi(-\vp(p)).
\end{align*}
Thus by Mertens theorem (see \cite{Nat} Theorem 6.6, page 160) we have
\begin{multline*}
  \sum_{p\in\mathbf{P}_{h, N}}\vp'(p)^{-1}\log p\ e^{2\pi i \xi W(p)}=\sum_{p\in\mathbf{P}_{N}} \vp'(p)^{-1}\big(\lfloor-\vp(p)\rfloor-\lfloor-\vp(p+1)\rfloor\big)\log p\ e^{2\pi i \xi W(p)}\\
  \ \ \ \ \ \ \ \ \  =\sum_{p\in\mathbf{P}_{N}} \log p\ e^{2\pi i \xi W(p)}
  +\sum_{p\in\mathbf{P}_{N}} \vp'(p)^{-1}\big(\Phi(-\vp(p+1))-\Phi(-\vp(p))\big)\log p\ e^{2\pi i \xi W(p)}
  +O(\log N).
\end{multline*}
The proof is completed since
\begin{align*}
  \sum_{p\in\mathbf{P}_{N}} &\vp'(p)^{-1}\big(\Phi(-\vp(p+1))-\Phi(-\vp(p))\big)\log p\ e^{2\pi i \xi W(p)}\\
  &=\sum_{k=1}^N \vp'(k)^{-1}\big(\Phi(-\vp(k+1))-\Phi(-\vp(k))\big)\Lambda(k) e^{2\pi i \xi W(k)}+
  O\big(N^{3/2-\g+\e}\big).
\end{align*}
In view of the assumptions of Lemma \ref{formlem} we see that $O\big(N^{3/2-\g+\e}\big)=O\big(N^{1-\chi-\e}\big)$, since
$3/2-\g+\e=1-\chi-\e+(2(1-\g)-1+2(2\e+\chi))/2<1-\chi-\e$ as desired.
\end{proof}
The proof of Lemma \ref{formlem} will be completed if we show that
\begin{align}\label{forma1}
  \sum_{k=1}^N \vp'(k)^{-1}\big(\Phi(-\vp(k+1))-\Phi(-\vp(k))\big)\Lambda(k)e^{2\pi i \xi W(k)}=O\big(N^{1-\chi-\chi'}\big),
\end{align}
for $\chi>0$  such that $(2^{2q+2}+2^q-2)(1-\g)+2^q(2^{q+3}-2)\chi<1$ and some $\chi'>0$. Expanding $\Phi$ into the Fourier series (see \cite{HB} Section 2), we obtain
\begin{align}\label{four1}
  \Phi(t)=\sum_{0<|m|\le M}\frac{1}{2\pi i m}e^{-2\pi imt}+O\left(\min\left\{1, {(M\|t\|)^{-1}}\right\}\right),
\end{align}
for $M>0$, where $\|t\|=\min_{n\in\Z}|t-n|$ is the distance of $t\in\R$ to the nearest integer. Moreover,
\begin{align}\label{four2}
  \min\left\{1, (M\|t\|)^{-1}\right\}=\sum_{m\in\Z}b_m e^{2\pi imt},
\end{align}
where
\begin{align}\label{fcoe2}
  |b_m|\lesssim \min\left\{M^{-1}{\log M}, |m|^{-1}, M|m|^{-2}\right\}.
\end{align}

\begin{lem}\label{lemest1}
Let $W:\Z\mapsto\Z$ be a fixed polynomial of degree $q\in\N$. Assume that $P\ge1$ and  $M=P^{1+\chi+\e}\vp(P)^{-1}$ with  $\chi>0$  such that $(2^{2q+2}+2^q-2)(1-\g)+2^q(2^{q+3}-2)\chi<1$ and $0<\e<\frac{\chi}{100(2^{q+2}-1)}$.  Then we have
\begin{align}\label{lem1eq1}
  &\sum_{P<k\le P_1\le 2P} \vp'(k)^{-1}\big(\Phi(-\vp(k+1))-\Phi(-\vp(k))\big)\Lambda(k)e^{2\pi i \xi W(k)}\\
  \nonumber=\sum_{0<|m|\le M}&\frac{1}{2\pi i m}\sum_{P<k\le P_1\le 2P} \vp'(k)^{-1}\Lambda(k) \Big(e^{2\pi i(\xi W(k)+m\vp(k+1))}-e^{2\pi i(\xi W(k)+m\vp(k))}\Big)+O\big(P^{1-\chi-\e}\big).
\end{align}
\end{lem}
\begin{proof}
Let $S$ be the first sum in \eqref{lem1eq1}, then  the Fourier expansion \eqref{four1} yields
\begin{align*}
  S
  &=\sum_{0<|m|\le M}\frac{1}{2\pi i m}\sum_{P<k\le P_1\le 2P} \vp'(k)^{-1}\Lambda(k) \Big(e^{2\pi i(\xi W(k)+m\vp(k+1))}-e^{2\pi i(\xi W(k)+m\vp(k))}\Big)\\
  &+O\bigg(\sum_{P<k\le P_1\le 2P}\vp'(k)^{-1}\Lambda(k)\big(\min\big\{1, ({M\|\vp(k)\|})^{-1}\big\}+\min\big\{1, ({M\|\vp(k+1)\|})^{-1}\big\}\big)\bigg).
\end{align*}
In a similar way as in \cite{M1} it suffices to bound the error term with $\min\big\{1, ({M\|\vp(k)\|})^{-1}\big\}$. The same reasoning will give the same bound for the sum with $\min\big\{1, ({M\|\vp(k+1)\|})^{-1}\big\}$. By \eqref{four2} we see that
\begin{align*}
 \sum_{P<k\le P_1\le 2P}\frac{\Lambda(k)}{\vp'(k)}\cdot\min\big\{1,({M\|\vp(k)\|})^{-1}\big\}
 \lesssim\frac{\log P}{\vp'(P)}\sum_{m\in\Z}|b_m|\bigg|\sum_{P<k\le P_1\le 2P} e^{2\pi im\vp(k)}\bigg|.
\end{align*}
Lemma \ref{vdclem2} applied to the inner sum with $l=1$, $j=0$  and $q=1$  and the bounds \eqref{fcoe2} for $|b_m|$ imply that
\begin{align*}
  \sum_{m\ge0}|b_m|\bigg|\sum_{P<k\le P_1\le 2P} e^{2\pi im\vp(k)}\bigg|\lesssim \frac{P\ \log M}{M} +
  \bigg(\sum_{0<m\le M}\frac{\log M}{M}+\sum_{m> M}\frac{M}{m^2}\bigg)\frac{m^{1/2}P}{\big(\s(P)\vp(P)\big)^{1/2}}\\
  \lesssim\frac{P\ \log M}{M}+\log M M^{1/2}\frac{P}{\big(\s(P)\vp(P)\big)^{1/2}}.
\end{align*}
Taking $M=P^{1+\chi+\e}\vp(P)^{-1}$,  we obtain
\begin{multline*}
\frac{\log P}{\vp'(P)} \sum_{m\ge0}|b_m|\bigg|\sum_{P<k\le P_1\le 2P} e^{2\pi im\vp(k)}\bigg|
\lesssim
 \frac{P\ \log M\ \log P}{\vp'(P)M}+\log M M^{1/2}\frac{P\log P}{\vp'(P)\big(\s(P)\vp(P)\big)^{1/2}}\\
 \lesssim \frac{\vp(P)P^{-\chi-\e}}{\vp'(P)}\ \log^2 P+ \frac{P^{3/2+\chi/2+\e/2}}{\vp'(P)\s(P)^{1/2}\vp(P)}\ \log^2 P\\
 \lesssim\frac{\vp(P)P^{-\chi-\e}}{\vp'(P)}\ \log^2 P\left(1+\frac{P^{3/2+3\chi/2+3\e/2}}{\s(P)^{1/2}\vp(P)^2}\right)
 \lesssim\frac{\vp(P)P^{-\chi-\e}}{\vp'(P)}\lesssim P^{1-\chi-\e}.
\end{multline*}
Taking $0<\e<\frac{\chi}{100(2^{q+2}-1)}<\chi/100$ one can show that the last parenthesis is bounded. Namely, due to the inequalities $x^{\g-\e}\lesssim_{\e}\vp(x)$, and $(\s(x))^{-1}\lesssim_{\e}x^{\e}$
which hold for arbitrary $\e>0$ we easily see that $3/2+3\chi/2+3\e/2+\e/2-2\g+2\e<0,$ since $3+3\chi+8\e-4\g<4(1-\g)+4\chi-1<(2^{2q+2}+2^q-2)(1-\g)+2^q(2^{q+3}-2)\chi-1<0$,
and this finishes the proof.
\end{proof}
Now we can illustrate the proof of Lemma \ref{formlem}.
\begin{proof}[Proof of Lemma \ref{formlem}]
Recall that   $\chi>0$  such that $(2^{2q+2}+2^q-2)(1-\g)+2^q(2^{q+3}-2)\chi<1$ and $0<\e<\frac{\chi}{100(2^{q+2}-1)}$. Then combining Lemma \ref{lemest0} with Lemma \ref{lemest1} we see that
\begin{multline}\label{estim1}
  \bigg|\sum_{p\in\mathbf{P}_{h, N}}\vp'(p)^{-1}\log p\ e^{2\pi i \xi W(p)}-\sum_{p\in\mathbf{P}_{N}} \log p\ e^{2\pi i \xi W(p)}\bigg|\\
  \lesssim\log N \sup_{1\le P\le N}\bigg|\sum_{P<k\le P_1\le 2P} \vp'(k)^{-1}\big(\Phi(-\vp(k+1))-\Phi(-\vp(k))\big)\Lambda(k)e^{2\pi i \xi W(k)}\bigg|+N^{1-\chi-\e}\\
  \lesssim\log N \sup_{1\le P\le N}\sum_{0<|m|\le M}\frac{1}{m}\bigg|\sum_{P<k\le P_1\le 2P} \vp'(k)^{-1}\Lambda(k) \Big(e^{2\pi i(\xi W(k)+m\vp(k+1))}-e^{2\pi i(\xi W(k)+m\vp(k))}\Big)\bigg|\\
  +N^{1-\chi-\e}\log N,
\end{multline}
where $M=P^{1+\chi+\e}\vp(P)^{-1}$. In order to bound the error term in \eqref{estim1} let us introduce $U_m(x)=\sum_{P< k\le x}\Lambda(k)e^{2\pi i(\xi W(k)+m\vp(k))}$, and  $\phi_m(k)=\vp'(k)^{-1}\big(e^{2\pi im(\vp(k+1)-\vp(k))}-1\big)$. It is easy to note that $|\phi_m(x)|\lesssim m$ and $|\phi_m'(x)|\lesssim\frac{m}{x}$. Therefore, summation by parts and the estimate \eqref{finbound} give


\begin{multline}\label{lastest}
\sum_{0<|m|\le M}\frac{1}{m}\bigg|\sum_{P<k\le P_1\le 2P} \vp'(k)^{-1}\Lambda(k) \Big(e^{2\pi i(\xi W(k)+m\vp(k+1))}-e^{2\pi i(\xi W(k)+m\vp(k))}\Big)\bigg|\\
 \lesssim \sum_{m=1}^M\frac{1}{m}\bigg(|U_m(P_1)\phi_m(P_1)|+\int_P^{P_1}|U_m(x)\phi_m'(x)|dx\bigg)
 \lesssim \sum_{m=1}^M\sup_{x\in(P, 2P]}|U_m(x)|\\
 \lesssim
 \sum_{m=1}^Mm^{\frac{1}{2^{q+1}-2}}\log^2 P\ \big(\s(P)\vp(P)\big)^{-\frac{1}{2^q}}P^{1+\frac{2^{q+1}-2}{2^{2q+1}+2^q-2}}\\
 +
 \sum_{m=1}^Mm^{\frac{1}{2^{q+2}-2}}\log^{6}P\ \big(\s(P)\vp(P)\big)^{-\frac{2^{q+1}-2}{2^q(2^{q+2}-2)}}P^{1+\frac{2^q-1}{2^q(2^{q+2}-2)}}\\
 \lesssim  M^{1+\frac{1}{2^{q+1}-2}}\log^2 P\ \big(\s(P)\vp(P)\big)^{-\frac{1}{2^q}}P^{1+\frac{2^{q+1}-2}{2^{2q+1}+2^q-2}}\\
 +  M^{1+\frac{1}{2^{q+2}-2}}\log^{6}P\ \big(\s(P)\vp(P)\big)^{-\frac{2^{q+1}-2}{2^q(2^{q+2}-2)}}P^{1+\frac{2^q-1}{2^q(2^{q+2}-2)}}.
\end{multline}
Now we have to estimate the last two terms in \eqref{lastest}. We will use the inequalities $x^{\g-\e}\lesssim_{\e}\vp(x)$, $\s(x)^{-1}\lesssim_{\e}x^{\e}$ and $\log x\lesssim_{\e}x^{\e/50}$ which hold with arbitrary $\e>0$.
Since $M=P^{1+\chi+\e}\vp(P)^{-1}$ with  $\chi>0$  such that $(2^{2q+2}+2^q-2)(1-\g)+2^q(2^{q+3}-2)\chi<1$ and $0<\e<\frac{\chi}{100(2^{q+2}-1)}$, it is easy to see  that
\begin{multline*}
  M^{1+\frac{1}{2^{q+1}-2}}\log^2 P\ \big(\s(P)\vp(P)\big)^{-\frac{1}{2^q}}P^{1+\frac{2^{q+1}-2}{2^{2q+1}+2^q-2}}\\
=\left(P^{1+\chi+\e}\vp(P)^{-1}\right)^{1+\frac{1}{2^{q+1}-2}}\log^2 P\ \big(\s(P)\vp(P)\big)^{-\frac{1}{2^q}}P^{1+\frac{2^{q+1}-2}{2^{2q+1}+2^q-2}}\\
  =P^{1+\chi+\e+\frac{\chi+\e}{2^{q+1}-2}+1+\frac{1}{2^{q+1}-2}+\frac{2^{q+1}-2}{2^{2q+1}+2^q-2}}
\vp(P)^{-1-\frac{1}{2^q}-\frac{1}{2^{q+1}-2}}\s(P)^{-\frac{1}{2^q}}\log^2 P\\
  \lesssim P^{1-\chi+10\e+\frac{(2^{q+2}-3)\chi}{2^{q+1}-2}+1+\frac{1}{2^{q+1}-2}+\frac{2^{q+1}-2}{2^{2q+1}+2^q-2}
-\g\left(1+\frac{1}{2^q}+\frac{1}{2^{q+1}-2}\right)}\\
\lesssim P^{1-\chi+1+\frac{1}{2^{q+1}-2}+\frac{2^{q+1}-2}{2^{2q+1}+2^q-2}+\frac{(2^{q+2}-2)\chi}{2^{q+1}-2}
-\g\left(\frac{2^{2q+1}+2^q-2}{2^q(2^{q+1}-2)}\right)}
\lesssim P^{1-\chi-\e'},
\end{multline*}
for some $\e'>0$, since $\log^2P\lesssim_{\e}P^{\e}$ and
\begin{multline*}
 1+\frac{1}{2^{q+1}-2}+\frac{2^{q+1}-2}{2^{2q+1}+2^q-2}+\frac{(2^{q+2}-2)\chi}{2^{q+1}-2}
-\g\left(\frac{2^{2q+1}+2^q-2}{2^q(2^{q+1}-2)}\right)<0\\
\Longleftrightarrow 2^q(2^{q+1}-2)+2^q+\frac{2^q(2^{q+1}-2)^2}{2^{2q+1}+2^q-2}+2^q(2^{q+2}-2)\chi<\g(2^{2q+1}+2^q-2)\\
\Longleftrightarrow(2^{2q+1}+2^q-2)(1-\g)+2^q(2^{q+2}-2)\chi-1+\frac{2^q(2^{q+1}-2)^2}{2^{2q+1}+2^q-2}-2^{q+1}+3<0\\
\Longleftrightarrow(2^{2q+1}+2^q-2)(1-\g)+2^q(2^{q+2}-2)\chi-1-\frac{(2^q-2)(2^{q+2}-3)}{2^{2q+1}+2^q-2}<0.
\end{multline*}
On the other hand, we get
\begin{multline*}
  M^{1+\frac{1}{2^{q+2}-2}}\log^{6}P\ \big(\s(P)\vp(P)\big)^{-\frac{2^{q+1}-2}{2^q(2^{q+2}-2)}}P^{1+\frac{2^q-1}{2^q(2^{q+2}-2)}}\\
  =\left(P^{1+\chi+\e}\vp(P)^{-1}\right)^{1+\frac{1}{2^{q+2}-2}}\log^{6}P\ \big(\s(P)\vp(P)\big)^{-\frac{2^{q+1}-2}{2^q(2^{q+2}-2)}}P^{1+\frac{2^q-1}{2^q(2^{q+2}-2)}}\\
  =P^{1+\chi+\e+\frac{\chi+\e}{2^{q+2}-2}+1+\frac{1}{2^{q+2}-2}+\frac{2^q-1}{2^q(2^{q+2}-2)}}
\vp(P)^{-1-\frac{1}{2^{q+2}-2}-\frac{2^{q+1}-2}{2^q(2^{q+2}-2)}}\s(P)^{-\frac{2^{q+1}-2}{2^q(2^{q+2}-2)}}\log^{6}P\\
  \lesssim P^{1-\chi+10\e+\frac{(2^{q+3}-3)\chi}{2^{q+2}-2}+1+\frac{1}{2^{q+2}-2}+\frac{2^q-1}{2^q(2^{q+2}-2)}
-\g\left(1+\frac{1}{2^{q+2}-2}+\frac{2^{q+1}-2}{2^q(2^{q+2}-2)}\right)}\\
\lesssim P^{1-\chi+\frac{2^{2q+2}-1}{2^q(2^{q+2}-2)}+\frac{(2^{q+3}-2)\chi}{2^{q+2}-2}
-\g\frac{2^{2q+2}+2^q-2}{2^q(2^{q+2}-2)}}\lesssim P^{1-\chi-\e'}.
\end{multline*}
for some $\e'>0$, since $\log^{6}P\lesssim_{\e}P^{\e}$ and
\begin{multline*}
  \frac{2^{2q+2}-1}{2^q(2^{q+2}-2)}+\frac{(2^{q+3}-2)\chi}{2^{q+2}-2}
-\g\frac{2^{2q+2}+2^q-2}{2^q(2^{q+2}-2)}<0\\
\Longleftrightarrow 2^{2q+2}-1+2^q(2^{q+3}-2)\chi<\g(2^{2q+2}+2^q-2)\\
\Longleftrightarrow (2^{2q+2}+2^q-2)(1-\g)+2^q(2^{q+3}-2)\chi-1-2^q+2<0.
\end{multline*}
This provides the desired upper bound for \eqref{lastest} and the proof of Lemma \ref{formlem} is completed.
\end{proof}

\section{Proof of Theorem \ref{maxithm}}\label{sectmax}
In this section, with the aid of Lemma \ref{formlem}, we shall illustrate the proof of Theorem \ref{maxithm}.
The maximal functions which will occur in this section will be initially defined for any nonnegative finitely supported function $f\ge0$ and unless otherwise stated $f$ is always such a function.
Let us introduce a maximal function
\begin{align}\label{maxh}
  \mathcal{M}_{h}f(x)=\sup_{N\in\N}|K_{h, N}*f(x)|\ \ \mbox{for  \  $x\in\Z$,}
\end{align}
corresponding with the kernel
\begin{align}\label{maxker}
  K_{h, N}(x)=\frac{1}{\pi_h(N)}\sum_{p\in\mathbf{P}_{h, N}}\d_{W(p)}(x) \ \ \mbox{for  \  $x\in\Z$,}
\end{align}
where $\d_n(x)$ denotes Dirac's delta at $n\in\N$ and $W:\Z\mapsto\Z$ is a fixed polynomial of degree $q\in\N$. Due to \eqref{asympto} we see that $\mathcal{M}_{h}f(x)\simeq M_{\mathbf{P}_h}f(x)$. Therefore, it suffices to show the inequality from \eqref{maxiest} with $\mathcal{M}_{h}f$ instead of $M_{\mathbf{P}_h}f$. Now we are going to slightly redefine the maximal function $\mathcal{M}_hf(x)$ defined in \eqref{maxh}. We will take the supremum over $\mathcal{D}=\{2^n: n\in\N\}$ rather that $\N$, i.e.
\begin{align}\label{maxh1}
  \mathcal{M}_hf(x)=\sup_{N\in\mathcal{D}}|K_{h, N}*f(x)|,
\end{align}
with $K_{h, N}$ defined in \eqref{maxker}. It will cause no confusion if we use the same letter $\mathcal{M}_{h}f$ in both definitions, since the maximal functions from \eqref{maxh} and \eqref{maxh1} are equivalent and give the same $\ell^r(\Z)$ bounds for $r>1$.

We start with some general observations concerning maximal functions.
\begin{lem}\label{estmax}
Let $S\subseteq\N$ be a fixed subset of integers and for $i=1, 2$, and $\Omega:\Z\mapsto\Z$ be a fixed function. Let us introduce
  \begin{itemize}
    \item  a nonnegative function $w_i:[0, \8)\mapsto[0, \8)$,
    \item  a sum $W_i(n)=\sum_{k\in S\cap[1, n]}w_i(k)$, corresponding with $w_i$, where $n\in\N$,
    \item  and a weighted maximal function
    $$\mathcal{M}_i^{\star}(f)(x)=\sup_{N\in Z}\frac{1}{W_i(N)}\Big|\sum_{k\in S\cap[1, N]}w_i(k)f(x-\Omega(k))\Big|,$$
  \end{itemize}
where $Z\subseteq\N$ and $f:\Z\mapsto\C$ is any nonnegative finitely supported function. Assume that
\begin{itemize}
  \item[(i)] the sequence $\left(\frac{w_2(n)}{w_1(n)}\right)_{n\in\N}$ is decreasing, or
  \item [(ii)] the sequence $\left(\frac{w_2(n)}{w_1(n)}\right)_{n\in\N}$ is increasing and $\sup_{n\in\N}\frac{w_2(n)\cdot W_1(n)}{w_1(n)\cdot W_2(n)}<\8$.
\end{itemize}
Then
$$\mathcal{M}_2^{\star}(f)(x)\lesssim \mathcal{M}_1^{\star}(f)(x),$$
for every $f\in\ell^r(\Z)$ and for every $x\in\Z$.
\end{lem}
Some variant of this lemma was proved in \cite{Wie}, but our formulation is more handy. We will apply Lemma \ref{estmax} with $Z=\N$ or $Z=\mathcal{D}$.
\begin{proof}
Assume that the sequence $\left(\frac{w_2(n)}{w_1(n)}\right)_{n\in\N}$ is increasing and $\sup_{n\in\N}\frac{w_2(n)\cdot W_1(n)}{w_1(n)\cdot W_2(n)}=C<\8$.  Denote $M_i^Nf(x)=\sum_{k\in S\cap[1, N]}w_i(k)f(x-\Omega(k))$ for $i=1, 2$ with $M_i^0f(x)=0$. Without loss of generality we may assume that $f\ge0$. Then applying summation by parts twice and exploring positive nature of the maximal operators we can easily observe that for every $N\in Z$ we have
\begin{multline*}
  \frac{1}{W_2(N)}M_2^Nf(x)=\frac{1}{W_2(N)}\sum_{k\in S\cap[1, N]}\frac{w_2(k)}{w_1(k)}w_1(k)f(x-\Omega(k))\\
  =\frac{1}{W_2(N)}\sum_{n=1}^N\frac{w_2(n)}{w_1(n)}(M_1^{n}f(x)-M_1^{n-1}f(x))=\frac{w_2(N)\cdot W_1(N)}{w_1(N)\cdot W_2(N)}\frac{1}{W_1(N)}M_1^{N}f(x)\\
  +\frac{1}{W_2(N)}\sum_{n=1}^{N-1}W_1(n)\left(\frac{w_2(n)}{w_1(n)}-\frac{w_2(n+1)}{w_1(n+1)}\right)
  \frac{1}{W_1(n)}M_1^{n}f(x)\\
  \le\mathcal{M}_1^{\star}(f)(x)\cdot\bigg(\frac{1}{W_2(N)}\sum_{n=1}^{N-1}W_1(n)\left(\frac{w_2(n+1)}{w_1(n+1)}
  -\frac{w_2(n)}{w_1(n)}\bigg)
  +\frac{w_2(N)\cdot W_1(N)}{w_1(N)\cdot W_2(N)}\right)\\
  =\mathcal{M}_1^{\star}(f)(x)\cdot\frac{w_2(N)\cdot W_1(N)}{w_1(N)\cdot W_2(N)}+\mathcal{M}_1^{\star}(f)(x)\cdot\frac{w_2(N)\cdot W_1(N-1)}{w_1(N)\cdot W_2(N)}\\
  +\mathcal{M}_1^{\star}(f)(x)\cdot\bigg(\frac{1}{W_2(N)}\sum_{n=2}^{N-1}\frac{w_2(n)}{w_1(n)}\left(W_1(n-1)
  -W_1(n)\right)
  -\frac{w_2(1)\cdot\mathbf{1}_{S}(1)}{W_2(N)}\bigg)\\
  =\mathcal{M}_1^{\star}(f)(x)\cdot\bigg(\frac{2w_2(N)\cdot W_1(N)}{w_1(N)\cdot W_2(N)}-\frac{w_2(N)\cdot\mathbf{1}_{S}(N)}{W_2(N)}-\frac{1}{W_2(N)}\sum_{n\in S\cap[2, N-1]}w_2(n)
  -\frac{w_2(1)\cdot\mathbf{1}_{S}(1)}{W_2(N)}\bigg)\\
  \le \mathcal{M}_1^{\star}(f)(x)\cdot\left(1+2\sup_{n\in\N}\frac{w_2(n)\cdot W_1(n)}{w_1(n)\cdot W_2(n)}\right)
  \le (1+2C)\mathcal{M}_1^{\star}(f)(x).
\end{multline*}
This implies that $\mathcal{M}_2^{\star}(f)(x)\lesssim \mathcal{M}_1^{\star}(f)(x)$ and the desired inequality follows. The proof when the sequence $\left(\frac{w_2(n)}{w_1(n)}\right)_{n\in\N}$ decreases is similar, but simpler, and is left to the reader.
\end{proof}

\begin{lem}\label{lempart1}
Let $W:\Z\mapsto\Z$ be a polynomial of degree $q\in\N$ and define a maximal function $\mathcal{M}_h^1f(x)=\sup_{N\in\mathcal{D}}|K_{h, N}^1*f(x)|,$ where
$$K_{h, N}^1(x)=\frac{1}{N}\sum_{p\in\mathbf{P}_{h, N}}\vp'(N)^{-1}\log p\ \delta_{W(p)}(x).$$
Then
$$\mathcal{M}_hf(x)\lesssim\mathcal{M}_h^1f(x),$$
for every $x\in\Z$.
\end{lem}
\begin{proof}
We shall apply Lemma \ref{estmax} to the maximal functions
$$\mathcal{M}_1^{\star}f(x)=\sup_{N\in\mathcal{D}}\frac{1}{W_1(N)}\sum_{p\in\mathbf{P}_{h, N}}w_1(p)f(x-W(p)),$$
$$\mathcal{M}_2^{\star}f(x)=\sup_{N\in\mathcal{D}}\frac{1}{W_2(N)}\sum_{p\in\mathbf{P}_{h, N}}w_2(p)f(x-W(p)),$$
 with weights $w_1(x)=\vp'(x)^{-1}\log x$, $w_2(x)=1$ and sums $W_1(N)=\sum_{p\in\mathbf{P}_{h, N}}\vp'(p)^{-1}\log p$ and
 $W_2(N)=\pi_h(N)$. Therefore, Lemma \ref{estmax} yields that
  $$\mathcal{M}_hf(x)=\mathcal{M}_2^{\star}f(x)\lesssim\mathcal{M}_1^{\star}f(x).$$
  What is left is to show that
  $$\mathcal{M}_1^{\star}f(x)=\sup_{N\in\mathcal{D}}\frac{1}{W_1(N)}\sum_{p\in\mathbf{P}_{h, N}}\vp'(p)^{-1}\log p\ f(x-W(p))\simeq \mathcal{M}_h^1f(x).$$
   If we prove $\frac{W_1(N)}{N}\ _{\overrightarrow{N\to\8}}\ 1$ the assertion follows.
   We now apply Lemma \ref{sbp} with $g(x)=\vp'(x)^{-1}\log x$ and $U(x)=\sum_{p\in\mathbf{P}_{h, x}}1=\pi_h(x)$.
   Indeed,
  \begin{align*}
    \frac{1}{N}\sum_{p\in\mathbf{P}_{h, N}}\vp'(p)^{-1}\log p=\frac{1}{N}\pi_h(N)\vp'(N)^{-1}\log N
    -\frac{1}{N}\int_{2}^N\pi_h(x)\big(\vp'(x)^{-1}\log x\big)'dx,
  \end{align*}
  Since
  \begin{align*}
    \lim_{N\to\8}\frac{1}{N}\pi_h(N)\vp'(N)^{-1}\log N=1/\g,
  \end{align*}
 it remains to prove that
  $$\frac{1}{N}\int_{2}^N\pi_h(x)\big(\vp'(x)^{-1}\log x\big)'dx=\frac{1-\g}{\g}.$$
  Observe that
  \begin{align*}
    \big(\vp'(x)^{-1}\log x\big)'=\frac{\vp'(x)-x\vp''(x)\log x }{x\vp'(x)^2}=\frac{1- \rho_h(x)\log x}{\vp(x)(\g+\te_1(x))},
  \end{align*}
  where
  \begin{align*}
    \rho_h(x)=\left\{ \begin{array} {ll}
\g-1+\te_2(x), & \mbox{if $\g<1$,}\\
\s(x)\t(x),& \mbox{if $\g=1$.}
\end{array}
\right.
  \end{align*}
  This easily shows that
  \begin{align*}
    \lim_{x\to\8}\frac{\vp(x)}{\log x}\cdot\frac{1- \rho_h(x)\log x}{\vp(x)(\g+\te_1(x))}
    =\left\{ \begin{array} {ll}
\frac{1-\g}{\g}, & \mbox{if $\g<1$,}\\
\ \ 0,& \mbox{if $\g=1$,}
\end{array}
\right.
  \end{align*}
  which in turn implies that
  \begin{align*}
    \lim_{N\to\8}\frac{1}{N}\int_{2}^N\pi_h(x)\big(\vp'(x)^{-1}\log x\big)'dx=
    \lim_{N\to\8}\frac{1}{N}\int_{2}^N\frac{1- \rho_h(x)\log x}{\log x(\g+\te_1(x))}dx,
  \end{align*}
  and consequently
   \begin{align*}
    \lim_{N\to\8}\frac{1}{N}\int_{2}^N\frac{1- \rho_h(x)\log x}{\log x(\g+\te_1(x))}dx=\frac{1-\g}{\g}.
  \end{align*}
  This completes the proof of the lemma.
\end{proof}
\noindent We have reduced the matters to proving
\begin{align}\label{ker2}
  \|\mathcal{M}_h^1f\|_{\ell^r(\Z)}\lesssim_r\|f\|_{\ell^r(\Z)},
\end{align}
for every $f\in\ell^r(\Z)$, where $r>1$.
For this purpose let us define $\mathcal{M}_h^2f(x)=\sup_{N\in\mathcal{D}}|K_{h, N}^2*f(x)|$, where
\begin{align*}
  K_{h, N}^2(x)=\frac{1}{N}\sum_{p\in\mathbf{P}_N}\log p\ \delta_{W(p)}(x).
\end{align*}
Due to Bourgain--Wierdl--Nair's theorem we know that
$$\|\mathcal{M}_h^2f\|_{\ell^r(\Z)}\lesssim_r\|f\|_{\ell^r(\Z)},$$
for every $f\in\ell^r(\Z)$, where $r>1$. For more details we refer to \cite{Na1, Na2}, see also \cite{B3}, \cite{Wie}. Observe now that
\begin{align*}
  \|\mathcal{M}_hf\|_{\ell^r(\Z)}\lesssim\|\mathcal{M}_h^1f\|_{\ell^r(\Z)}&\le \|\mathcal{M}_h^2f\|_{\ell^r(\Z)}+\|\sup_{N\in\mathcal{D}}|(K_{h, N}^1-K_{h, N}^2)*f|\|_{\ell^r(\Z)}\\
  &\lesssim \|f\|_{\ell^r(\Z)}+\bigg\|\bigg(\sum_{N\in\mathcal{D}}|(K_{h, N}^1-K_{h, N}^2)*f|^r\bigg)^{1/r}\bigg\|_{\ell^r(\Z)}\\
  &\lesssim \|f\|_{\ell^r(\Z)}+\bigg(\sum_{N\in\mathcal{D}}\|(K_{h, N}^1-K_{h, N}^2)*f\|_{\ell^r(\Z)}^r\bigg)^{1/r}.
\end{align*}
The estimate \eqref{ker2} will be completed if we establish the following inequality
\begin{align*}
  \bigg(\sum_{N\in\mathcal{D}}\|(K_{h, N}^1-K_{h, N}^2)*f\|_{\ell^r(\Z)}^r\bigg)^{1/r}\lesssim \|f\|_{\ell^r(\Z)}.
\end{align*}
For this purpose we begin with $r=2$ and apply Plancherel theorem to $K_{h, N}^1-K_{h, N}^2$ which allows us to make use of Lemma \ref{formlem}.
Indeed,
\begin{align*}
  \|(K_{h, N}^1-K_{h, N}^2)*f\|_{\ell^2(\Z)}^2&=\int_0^1|\widehat{K}_{h, N}^1(\xi)-\widehat{K}_{h, N}^2(\xi)|^2\cdot|\widehat{f}(\xi)|^2d\xi\\
  &\le\sup_{\xi\in[0, 1]}|\widehat{K}_{h, N}^1(\xi)-\widehat{K}_{h, N}^2(\xi)|^2\int_0^1|\widehat{f}(\xi)|^2d\xi\\
  &=\|\widehat{K}_{h, N}^1-\widehat{K}_{h, N}^2\|_{L^{\8}([0, 1])}^2\|f\|_{\ell^2(\Z)}^2\lesssim N^{-2\chi }\|f\|_{\ell^2(\Z)}^2,
\end{align*}
since by Lemma \ref{formlem} we have
$$\|\widehat{K}_{h, N}^1-\widehat{K}_{h, N}^2\|_{L^{\8}([0, 1])}\lesssim N^{-\chi },$$
for some $\chi>0$ and every $N\in\mathcal{D}$. Now observe that
$$\|(K_{h, N}^1-K_{h, N}^2)*f\|_{\ell^r(\Z)}\le\|K_{h, N}^1-K_{h, N}^2\|_{\ell^1(\Z)}\|f\|_{\ell^r(\Z)}\lesssim\|f\|_{\ell^r(\Z)}.$$
Since from the proof of Lemma \ref{lempart1} we know
\begin{align*}
  \|K_{h, N}^1\|_{\ell^1(\Z)}&=\sum_{x\in \Z}\frac{1}{N}\sum_{p\in\mathbf{P}_{h, N}}\vp'(p)^{-1}\log p\ \d_{W(p)}(x)\\
&=\frac{1}{N}\sum_{p\in\mathbf{P}_{h, N}}\vp'(p)^{-1}\log p\ _{\overrightarrow{N\to\8}}\ 1,
\end{align*}
and the same reasoning applies to
$$\|K_{h, N}^2\|_{\ell^1(\Z)}=\frac{1}{N}\sum_{p\in\mathbf{P}_N}\log p\ _{\overrightarrow{N\to\8}}\ 1.$$
Therefore, Riesz--Thorin interpolation theorem yields that for every $1<r\le2$ there is $\chi_r>0$ such that
$$\|(K_{h, N}^1-K_{h, N}^2)*f\|_{\ell^r(\Z)}\lesssim N^{-\chi_r}\|f\|_{\ell^r(\Z)},$$
for every $N\in\mathcal{D}$. Finally we obtain the desired bounds
$$\bigg(\sum_{N\in\mathcal{D}}\|(K_{h, N}^1-K_{h, N}^2)*f\|_{\ell^r(\Z)}^r\bigg)^{1/r}\lesssim
\bigg(\sum_{N\in\mathcal{D}}N^{-\chi_rr}\bigg)^{1/r}\|f\|_{\ell^r(\Z)}\lesssim_r\|f\|_{\ell^r(\Z)}.$$
This completes the proof of Theorem \ref{maxithm}.

\section{Proof of Theorem \ref{ergthm}}\label{secterg}
The main aim of this section is to prove Theorem \ref{ergthm}. For this purpose we will proceed as follows.
First of all we show the pointwise convergence on $L^2(X,\mu)$, then we can easily extend the pointwise convergence of $A_{h, N}f(x)$ for all $f\in L^r(X,\mu)$, where $r>1$. We start with very simple observation based on summation by parts. Namely, if
\begin{align}\label{ergcon1}
  A_{h, N}^1f(x)=\frac{1}{N}\sum_{p\in\mathbf{P}_{h, N}} \vp'(p)^{-1}\log p\ f(T^{W(p)}x) \ _{\overrightarrow{N\to\8}}\ f^*(x) \ \ \mbox{for $\mu$ -- a.e. $x\in X$,}
\end{align}
then
\begin{align}\label{ergcon2}
 A_{h, N}f(x)=\frac{1}{\pi_h(N)}\sum_{p\in\mathbf{P}_{h, N}}f(T^{W(p)}x) \ _{\overrightarrow{N\to\8}}\ f^*(x) \ \ \mbox{for $\mu$ -- a.e. $x\in X$.}
\end{align}
Let $M_kf(x)=\sum_{p\in\mathbf{P}_{h, k}}f(T^{W(p)}x)$ and $M_k^1f(x)=\sum_{p\in\mathbf{P}_{h, k}}\vp'(p)^{-1}\log p\ f(T^{W(p)}x)$ and $M_1f(x)=M_1^1f(x)=0$. Let $m_k=\sum_{p\in\mathbf{P}_{h, k}}1=\pi_h(k)$ and
$m_k^1=\sum_{p\in\mathbf{P}_{h, k}} \vp'(p)^{-1}\log p$. Then, for $f\ge0$, we have
\begin{align*}
  \frac{1}{\pi_h(N)}\sum_{p\in\mathbf{P}_{h, N}}f(T^{W(p)}x)&=
  \frac{1}{\pi_h(N)}\sum_{p\in\mathbf{P}_{h, N}}\frac{\vp'(p)}{\log p}\vp'(p)^{-1}\log p\ f(T^{W(p)}x)\\
  &=\frac{1}{\pi_h(N)}\sum_{k=2}^N\frac{\vp'(k)}{\log k}\big(M_k^1f(x)-M_{k-1}^1f(x)\big)\\
  &=\frac{\vp'(N)}{\pi_h(N)\log N}M_N^1f(x)+
  \frac{1}{\pi_h(N)}\sum_{k=2}^{N-1}\left(\frac{\vp'(k)}{\log k}-\frac{\vp'(k+1)}{\log(k+1)}\right)M_k^1f(x)\\
  &=\frac{m_N^1\vp'(N)}{\pi_h(N)\log N}\frac{N}{m_N^1}A_{h, N}^1f(x)\\
  &+
  \frac{1}{\pi_h(N)}\sum_{k=2}^{N-1}\left(\frac{m_k^1\vp'(k)}{\log k}-\frac{m_k^1\vp'(k+1)}{\log(k+1)}\right)\frac{k}{m_k^1}A_{h, k}^1f(x).
\end{align*}
On the other hand
\begin{align*}
  \frac{m_N^1\vp'(N)}{\pi_h(N)\log N}f^*(x)+
  \frac{1}{\pi_h(N)}\sum_{k=2}^{N-1}\left(\frac{m_k^1\vp'(k)}{\log k}-\frac{m_k^1\vp'(k+1)}{\log(k+1)}\right)f^*(x)\\
  =\frac{1}{\pi_h(N)}\sum_{k=2}^{N}\frac{\vp'(k)}{\log k}\big(m_k^1-m_{k-1}^1\big)f^*(x)=f^*(x).
\end{align*}
Let $\e>0$ such that for every $N> N_0$ we have
$$\left|\frac{N}{m_N^1}A_{h, N}^1f(x)-f^*(x)\right|<\e/4.$$
Since, $\pi_h(N)\ _{\overrightarrow{N\to\8}}\ \8$, we see
\begin{multline*}
  \limsup_{N\to\8}\bigg|\frac{1}{\pi_h(N)}\sum_{p\in\mathbf{P}_{h, N}}f(T^{W(p)}x)-f^*(x)\bigg|
  \le\limsup_{N\to\8}\frac{m_N^1\vp'(N)}{\pi_h(N)\log N}\left|\frac{N}{m_N^1}A_{h, k}^1f(x)-f^*(x)\right|\\
  +
  \limsup_{N\to\8}\bigg(\frac{1}{\pi_h(N)}\bigg(\sum_{k=2}^{N_0}+\sum_{k=N_0+1}^{N-1}\bigg)
\left(\frac{m_k^1\vp'(k)}{\log k}-\frac{m_k^1\vp'(k+1)}{\log(k+1)}\right)\left|\frac{k}{m_k^1}A_{h, k}^1f(x)-f^*(x)\right|\bigg)
  \le \e,
\end{multline*}
and \eqref{ergcon2} is justified.

In order to prove \eqref{ergcon1} on $L^2(X, \mu)$ it suffices to show that
\begin{align}\label{ergmax}
  \big\|\sup_{N\in\N}|A_{h, N}^1f|\big\|_{L^2(X, \mu)}\lesssim\|f\|_{L^2(X, \mu)},
\end{align}
and
\begin{align}\label{ergosc}
  \sum_{j=1}^J\bigg\|\sup_{\genfrac{}{}{0pt}{}{N_j<N\le N_{j+1}}{N\in Z_{\e}}}|A_{h, N}^1f-A_{h, N_j}^1f|\bigg\|_{L^2(X, \mu)}\le o(J)\|f\|_{L^2(X, \mu)},
\end{align}
where $Z_{\e}=\{\lfloor(1+\e)^n\rfloor: n\in\N\}$ for some fixed $\e>0$ and $(N_j)_{j\in\N}$ is any rapidly
increasing sequence $2N_j<N_{j+1}$. Using transference principle as in \cite{B3} we see that \eqref{ergmax}
and \eqref{ergosc} can be transferred to $\Z$ and \eqref{ergmax} is equivalent to \eqref{ker2} (with $r=2$) from Section \ref{sectmax}. If it comes to \eqref{ergosc} we use Lemma \ref{formlem}. Indeed,
\begin{align*}
  \sum_{j=1}^J\bigg\|\sup_{\genfrac{}{}{0pt}{}{N_j<N\le N_{j+1}}{N\in Z_{\eps}}}|K^1_{h, N}*f-K^1_{h, N_j}*f|\bigg\|_{\ell^2(\Z)}
  &\lesssim \sum_{j=1}^J\bigg\|\sup_{\genfrac{}{}{0pt}{}{N_j<N\le N_{j+1}}{N\in Z_{\eps}}}|K^2_{h, N}*f-K^2_{h, N_j}*f|\bigg\|_{\ell^2(\Z)}\\
  &+\sum_{j=1}^J\bigg(\sum_{\genfrac{}{}{0pt}{}{N_j<N\le N_{j+1}}{N\in Z_{\eps}}}\left\|K^1_{h, N}*f-K^2_{h, N}*f\right\|^2_{\ell^2(\Z)}\bigg)^{1/2}\\
  &\lesssim o(J)\|f\|_{\ell^2(\Z)}+\sum_{j=1}^J N_j^{-\chi}\|f\|_{\ell^2(\Z)}\lesssim o(J)\|f\|_{\ell^2(\Z)},
\end{align*}
as desired. Since the first inequality follows from  \cite{Na2}, and the second one follows by Parseval's identity and Lemma \ref{formlem}. This completes the proof of Theorem \ref{ergthm}.

\section{Proof of Theorem \ref{asymptthm}}\label{sectasym}
This section is intended to prove Theorem \ref{asymptthm}. In particular we will be concerned with
showing \eqref{asymptineq}.  We shall apply  \eqref{form} with $W(x)=x$, $q=1$ and any $0<\g\le 1$ and $\chi>0$ such that $16(1-\g)+28\chi<1$. Let $R(N)$ be the number of representations of an odd $N\in\N$ as a sum of three primes $p_i\in\mathbf{P}_{h_i}$
where $i=1, 2, 3$. Let
\begin{align*}
  \widetilde{G}^i_N(\xi)=\sum_{p\in\mathbf{P}_{h_i, N}}\vp_i'(p)^{-1}\log p\ e^{2\pi i \xi p},\ \ \ \mbox{and}\ \ \ \widetilde{F}^i_N(\xi)=\sum_{p\in\mathbf{P}_N}\log p\ e^{2\pi i \xi p},
\end{align*}
for $i=1, 2, 3$. Lemma \ref{sbp} applied twice and Lemma \ref{formlem} yield
\begin{align*}
  G^i_N(\xi)&=\sum_{p\in\mathbf{P}_{h_i,N}}e^{2\pi i \xi p}=\sum_{p\in\mathbf{P}_{h_i,N}}\left(\vp_i'(p)^{-1}\log p\ e^{2\pi i \xi p}\right)\frac{\vp_i'(p)}{\log p}\\
  &
  =\widetilde{G}^i_N(\xi)\frac{\vp_i'(N)}{\log N}-\int_2^N\widetilde{G}^i_x(\xi)\left(\frac{\vp_i'(x)}{\log x}\right)'dx\\
  &=\widetilde{F}^i_N(\xi)\frac{\vp_i'(N)}{\log N}-\int_2^N\widetilde{F}^i_x(\xi)\left(\frac{\vp_i'(x)}{\log x}\right)'dx+O\left(\vp_i(N)N^{-\chi_i-\e_i'}\right)\\
  &=\sum_{p\in\mathbf{P}_N}\vp_i'(p)\ e^{2\pi i \xi p}+O\left(\vp_i(N)N^{-\chi_i-\e_i'}\right)=F^i_N(\xi)+O\left(\vp_i(N)N^{-\chi_i-\e_i'}\right),
\end{align*}
for some $\e_i'>0$ and $\chi_i>0$ as in Lemma \ref{formlem}.

Thus
\begin{align*}
  R(N)=\sum_{\genfrac{}{}{0pt}{}{p_1+p_2+p_3=N}{p_i\in\mathbf{P}_{h_i, N}}}1&=
  \int_0^1G^1_N(\xi)G^2_N(\xi)G^3_N(\xi)e^{-2\pi i \xi N}d\xi
  =\sum_{\genfrac{}{}{0pt}{}{p_1+p_2+p_3=N}{p_i\in\mathbf{P}_N}}\vp_1'(p_1)\vp_2'(p_2)
  \vp_3'(p_3)\\
  &+\int_0^1\left(G^1_N(\xi)G^2_N(\xi)G^3_N(\xi)-F^1_N(\xi)F^2_N(\xi)F^3_N(\xi)\right)e^{-2\pi i \xi N}d\xi.
\end{align*}
Analysis similar to that in the proof of Theorem 1 in \cite{BF} shows that
\begin{multline}\label{as1}
  \left|\int_0^1\left(G^1_N(\xi)G^2_N(\xi)G^3_N(\xi)-F^1_N(\xi)F^2_N(\xi)F^3_N(\xi)\right)e^{-2\pi i \xi N}d\xi\right|\\
  \le\sup_{\xi\in[0, 1]}\left|G^1_N(\xi)-F^1_N(\xi)\right|\left(\int_0^1|F^2_N(\xi)|^2d\xi\right)^{1/2}
  \left(\int_0^1|F^3_N(\xi)|^2d\xi\right)^{1/2}\\
  +\sup_{\xi\in[0, 1]}\left|G^2_N(\xi)-F^2_N(\xi)\right|\left(\int_0^1|G^1_N(\xi)|^2d\xi\right)^{1/2}
  \left(\int_0^1|F^3_N(\xi)|^2d\xi\right)^{1/2}\\
  +\sup_{\xi\in[0, 1]}\left|G^3_N(\xi)-F^3_N(\xi)\right|\left(\int_0^1|G^1_N(\xi)|^2d\xi\right)^{1/2}
  \left(\int_0^1|G^2_N(\xi)|^2d\xi\right)^{1/2}.
\end{multline}
By Parseval's identity we see
\begin{align}\label{as2}
  \left(\int_0^1|F^i_N(\xi)|^2d\xi\right)^{1/2}=\bigg(\sum_{p\in\mathbf{P}_N}\vp_i'(p)^2\bigg)^{1/2}
  \le\bigg(\sum_{p\in\mathbf{P}_N}\vp_i'(p)\bigg)^{1/2}\lesssim\left(\frac{\vp_i(N)}{\log N}\right)^{1/2},
\end{align}
and
\begin{align}\label{as3}
  \left(\int_0^1|G^i_N(\xi)|^2d\xi\right)^{1/2}=\bigg(\sum_{p\in\mathbf{P}_{h_i, N}}1\bigg)^{1/2}
  \lesssim\left(\frac{\vp_i(N)}{\log N}\right)^{1/2},
\end{align}
for every $i=1, 2, 3.$
Combining \eqref{as1}, \eqref{as2}, \eqref{as3} with Lemma \ref{formlem} we immediately obtain
\begin{align}\label{err1}
&\left|\int_0^1\left(G^1_N(\xi)G^2_N(\xi)G^3_N(\xi)-F^1_N(\xi)F^2_N(\xi)F^3_N(\xi)\right)e^{-2\pi i \xi N}d\xi\right|\\
\nonumber&\lesssim\frac{\vp_1(N)\vp_2(N)\vp_3(N)}{N\ \log^3 N}\left(\frac{N^{1-\chi_1-\e_1'}\log^2 N}{\vp_2(N)^{1/2}\vp_3(N)^{1/2}}+\frac{N^{1-\chi_2-\e_2'}\log^2 N}{\vp_1(N)^{1/2}\vp_3(N)^{1/2}}
+\frac{N^{1-\chi_3-\e_3'}\log^2 N}{\vp_1(N)^{1/2}\vp_2(N)^{1/2}}\right).
\end{align}
If at least two of $\g_1, \g_2, \g_3$ are not equal to $1$, then bearing in mind \eqref{asymptthmass} and the inequality $x^{\g_i-\d_i}\lesssim\vp_i(x)$ one can arrange $\chi_1=\frac 1 2 (1-\g_2)+\frac 1 2 (1-\g_3)>0$, $\chi_2=\frac 1 2 (1-\g_1)+\frac 1 2 (1-\g_3)>0$, $\chi_3=\frac 1 2 (1-\g_1)+\frac 1 2 (1-\g_2)>0$ and $\frac 1 2(\d_1+\d_2+\d_3)<\min\{\e_1', \e_2', \e_3'\}$, and observe that
\begin{align}\label{err2}
\frac{N^{1-\chi_1-\e_1'}\log^2 N}{\vp_2(N)^{1/2}\vp_3(N)^{1/2}}\lesssim
 N^{1-\chi_1-\e_1'-\frac{\g_2+\g_3}{2}+\frac{\d_2+\d_3}{2}}\log^2N\lesssim N^{-\d'},
\end{align}
for some $\d'>0$. The two remaining summands in \eqref{err1} have  decays of the same type.
If for every $k, l\in\{1, 2, 3\}$ we have $\g_k=1$ or $\g_l=1$, then there are two, let say $\g_2, \g_3$ such that $\g_2=1$ and $\g_3=1$, then we have the same bound as in \eqref{err2} with $\chi_1>0$ carried by Lemma \ref{formlem}.
 Therefore,
  \begin{align*}
    R(N)=\sum_{\genfrac{}{}{0pt}{}{p_1+p_2+p_3=N}{p_i\in\mathbf{P}_{h_i, N}}}1
  =\sum_{\genfrac{}{}{0pt}{}{p_1+p_2+p_3=N}{p_i\in\mathbf{P}_N}}\vp_1'(p_1)\vp_2'(p_2)
  \vp_3'(p_3)+o\left(\frac{\vp_1(N)\vp_3(N)\vp_3(N)}{N\ \log^3 N}\right).
  \end{align*}
  We are now in a position where we can easily derive \eqref{asymptineq}. Let $r(N)$ be the number of representations of an odd $N\in\N$ as a sum of three regular primes, i.e. $p_i\in\mathbf{P}$
where $i=1, 2, 3$. Then Vinogradov's theorem (see \cite{Nat} Chapter 8) provides an asymptotic formula
\begin{align*}
  r(N)=\sum_{\genfrac{}{}{0pt}{}{p_1+p_2+p_3=N}{p_i\in\mathbf{P}_{N}}}1=\frac{\mathfrak{S}(N)N^2}{2\log^3N}
  +o\left(\frac{N^2}{\log^3N}\right),
\end{align*}
where $\mathfrak{S}(N)$ is the singular series. Therefore, there exists $C=C(\g_1, \g_2,\g_3)>0$ such that
  \begin{align*}
    \sum_{\genfrac{}{}{0pt}{}{p_1+p_2+p_3=N}{p_i\in\mathbf{P}_N}}\vp_1'(p_1)\vp_2'(p_2)
  \vp_3'(p_3)\ge\vp_1'(N)\vp_2'(N)
  \vp_3'(N)\cdot r(N)
  \ge C\cdot \frac{\mathfrak{S}(N)\vp_1(N)\vp_2(N)
  \vp_3(N)}{N\log^3N},
  \end{align*}
for sufficiently large $N\in\N$. The proof of Theorem \ref{asymptthm} is completed.

\end{document}